\documentclass[11pt, reqno]{amsart}
\usepackage{latexsym,amsthm,amsmath,amssymb,enumerate,graphicx, color}
\usepackage{epstopdf}
\usepackage[backref = true]{hyperref}
\usepackage[margin=1in]{geometry}
\author{Nakul Dawra}
\title{On the link Floer homology of $L$-space links}
\date{\today}
\newtheorem{thm}{Theorem}[section]
\newtheorem{ack}[thm]{Acknowledgments}
\newtheorem{prop}[thm]{Proposition}
\newtheorem{lem}[thm]{Lemma}
\newtheorem{cor}[thm]{Corollary}
\theoremstyle{definition}
\newtheorem{defn}[thm]{Definition}

\newtheorem{exmp}[thm]{Example}
\newtheorem{rem}[thm]{Remark}
\theoremstyle{remark}
\newtheorem{lem*}{Lemma}

\newtheorem{case}{Case}

\input xy
\xyoption{all}
\begin{document}
\begin{abstract}
We will prove that, for a $2$ or $3$ component $L$-space link, $HFL^-$ is completely determined by the multi-variable Alexander polynomial of all the sub-links of $L$, as well as the pairwise linking numbers of all the components of $L$. We will also give some restrictions on the multi-variable Alexander polynomial of an $L$-space link. Finally, we use the methods in this paper to prove a conjecture by Yajing Liu classifying all $2$-bridge $L$-space links.
\end{abstract}
\maketitle

\section{Introduction}\label{sec:Introduction}
In \cite{OZKnot} and \cite{OZLinks}, knot and link Floer homology were defined as a part of Ozsv\'{a}th and Szab\'{o}'s Heegaard Floer theory (introduced in \cite{OZOrig}). These give rise to graded homology groups, which are invariants of isotopy classes of knots and links embedded in $S^3$. Carefully examining these groups has yielded a wealth of topological insights (see \cite{NFibered}, \cite{OZ4Genus}, \cite{OZGenus}, \cite{OZTnorm}, \cite{OZRational} and \cite{WCosmetic}). The Euler characteristic of link (knot) Floer homology is the multi-variate (single variable) Alexander polynomial\footnote{ This is ``almost true''. We will make it precise in Definition \ref{def:Apoly}.}. \\

Throughout this paper, we will work over the field $\mathbb{F} = \mathbb{Z}/2\mathbb{Z}$, and $L = L_1 \sqcup L_2 \sqcup \ldots \sqcup L_l$ will always be an $l$ component link inside $S^3$ unless otherwise specified. We will focus on links all of whose large positive surgeries yeild $L$-spaces.\\

$L$-spaces are rational homology spheres whose Heegaard Floer homology is the simplest possible. More specifically, recall that for any rational homology $3$-sphere $Y$ we must have dim$(\widehat{HF}(Y))\geq |H_1(Y)|$, and so we define an $L$-space as:
\begin{defn}  
$Y$ a $\mathbb{Q}HS^3$ is an $L$-space if dim$(\widehat{HF}(Y)) = |H_1(Y)|$.
\end{defn}
Lens spaces are the simplest examples of $L$-spaces. Further examples include any connected sums of $3$-manifold with elliptic geometry\cite{OZLspace}, as well as double branched covers of quasi-alternating links \cite{OZDbcovers}. It was shown in Theorem $1.4$ of \cite{OZOrig} that such manifolds do not admit co-orientable $C^2$ taut foliations.\\
We will define an $L$-space link as follows:
\begin{defn}
$L \subset S^3$ is an $L$-space link if the $3$-manifolds $S^3_{n_1,\ldots,n_l}(L)$ obtained by surgery on $L$ are all $L$-spaces when all of the $n_i$ are sufficiently large.\footnote{ Note that this definition does not depend on the orientation of the components of $L$.}
\end{defn}

$L$-space links were first studied in \cite{GNAlgebraic}, where it was shown that any link arising as the embedded link of a complex plane curve singularity (i.e. algebraic link) is an $L$-space link (note that this includes all torus links). The general study of properties and examples of $L$-space links was initiated in \cite{YLspace} see also \cite{HLspace}. $L$-space knots were first examined in \cite{OZLspace}. In that paper it was shown that for an $L$-space knot, the knot Floer homology is completely determined by its Euler Characteristic (i.e. the Alexander polynomial). In this paper, we give a generalization of this statement to $2$ and $3$ component $L$-space links inside $S^3$. First, we recall some standard facts and notation.
\begin{defn}
Let $\mathbb{H}(L)_i$ denote the affine lattice over $\mathbb{Z}$ given by $\mbox{lk}(L_i, L\char92L_i)/2 + \mathbb{Z}$. We define:
\[\mathbb{H}(L) := \bigoplus_{i=1}^l \mathbb{H}(L)_i.\]
We can think of every element of $\mathbb{H}(L)$ as an element of the set of relative Spin$^c$ structures of $L\subset S^3$ via the identification $\mathbb{H}(L) \to \underline{\mbox{Spin}^c}(S^3,L)$ given in section 8.1 of \cite{OZLinks}. Note that $\mathbb{H}(L)$ is an affine lattice over $H_1(S^3 - L) \cong \mathbb{Z}^l$.
\end{defn}

Both $HFL^-$ and $\widehat{HFL}$ for a link $L$ inside $S^3$ split into direct summands indexed by pairs $(d,\mbox{\textbf{s}})$, where $d\in \mathbb{Z}$ (the homological grading) and \textbf{s} $\in \mathbb{H}(L)$. We will write these summands as $HFL^-_d(L,\mbox{\textbf{s}})$ and $\widehat{HFL}_d(L,\mbox{\textbf{s}})$.\\

Now, if \textbf{s} $= (s_1,s_2,\ldots, s_l) \in \mathbb{H}(L)$, we denote by $u^{\mbox{\scriptsize{\textbf{s}}}}$ the monomial $u_1^{s_1}\ldots u_l^{s_l}$. 
\begin{defn}\label{def:Apoly}
In this paper, we define the symmetric multi-variable Alexander polynomial $\Delta_L(u_1,u_2,\ldots, u_l)$ for $L$ so that the following equality\footnote{ In proposition $9.1$ of \cite{OZLinks}, the above equality was only shown to hold up to sign. So our sign convention for $\Delta_L$ here may not be standard, but it will make the statement of some of our Theorems easier. For our main Theorem, we only need to know $\Delta_L$ up to sign.} holds:
\[\sum_{\mbox{\scriptsize{\textbf{s}}} \in \mathbb{H}(L)} \chi(\widehat{HFL}_*(L,\mbox{\textbf{s}}))u^{\mbox{\scriptsize{\textbf{s}}}} =  
\prod_{i=1}^l\left(u_i^{1/2} - u_i^{-1/2}\right)\Delta_L(u_1,u_2,\ldots,u_l).\]
\end{defn}

\begin{thm}\label{thm:main} Let $L \subset S^3$ be a $2$ or $3$ component $L$-space link and let \textnormal{\textbf{s}} $\in \mathbb{H}(L)$. Then $HFL^-(L,\mbox{\textnormal{\textbf{s}}})$ is completely determined by the symmetric multi-variable Alexander polynomials $\pm\Delta_M$ for every sub-link $M \subset L$, as well as the pairwise linking numbers of components of $L$.
\end{thm}

In \cite{OZLspace}, it was shown that being an $L$-space knot forces strong restrictions on the Alexander polynomial, and we will generalize this to links. Our restrictions will depend on the Alexander polynomial of the link $L$, as well as the Alexander polynomial of all its sub-links after a shift depending on various linking numbers.
\begin{defn} 
Given a proper subset $S = \{i_1,i_2, \ldots, i_k\} \subsetneq \{1,\ldots, l\}$, we let $\{j_1, j_2, \ldots, j_{l-k}\} = \{1,\ldots, l\}\char92 S$ where $j_a < j_b$ when $a<b$. Let $L_S \subset L$ be the sub-link $L_{i_1}\sqcup L_{i_2} \sqcup \ldots \sqcup L_{i_k}$. The polynomial $P^L_{L_S}$ is defined as follows:\\
When $S = \emptyset$ we have,
\[P^L_\emptyset = \left(\prod_{i=1}^l u_i^{1/2}\right)\Delta_L(u_1,\ldots,u_l);\]
When $l-k > 1$ we have,
\[P^L_{L_S}(u_{j_1},u_{j_2},\ldots,u_{j_{l-k}}) = \left(\prod_{p=1}^{l-k} u_{j_i}^{1/2 + \mbox{\scriptsize{\textnormal{lk}}}(L_{j_i}, L_S)/2}\right)\Delta_{L\char92 L_S}(u_{j_1},\ldots, u_{j_{(l-k)}});\]
And finally when $l-k = 1$ we have,
\[P^L_{L_S}(u_{j_1}) = u_{j_1}^{\frac{\mbox{\scriptsize{\textnormal{lk}}}\left(L_{j_1}, L_S\right)}{2}}\left(\sum_{i\geq 0} u_{j_1}^{-i}\right)\Delta_{L\char92 L_S}(u_{j_1}).\]

Now, fix some \textnormal{\textbf{s}} $ = (s_1,s_2,\ldots, s_l)\in \mathbb{H}(L)$ and $r \in \{1,\ldots, l\}$ so that $r \not\in S$. Then, define 
\[R_{\substack{\mbox{\scriptsize{\textbf{s}}}'\geq \mbox{\scriptsize{\textbf{s}}}\\s'_r = s_r}}(P^L_{L_S})\]
to be the sum of all the coefficients of monomials $u_{j_1}^{s'_1}\ldots u_{j_{l-k}}^{s'_{j_{l-k}}}$ of $P^L_{L_S}$ that satisfy $s'_r = s_r$ and $s'_{j_p} \geq s_{j_p}$ for $j_p \neq r$. 
\end{defn}

\begin{exmp} Consider the $2$-bridge link $L = b(20,-3)$ (see Section \ref{sec:Application} for definitions and notation). Then;
\begin{align*}
\Delta_L(u_1,u_2) =& u_1^{1/2}u_2^{3/2} + u_1^{3/2}u_2^{1/2} + u_1^{1/2}u_2^{-1/2} + u_1^{-1/2}u_2^{1/2} +u_1^{-3/2}u_2^{-1/2} + u_1^{-1/2}u_2^{-3/2}- u_1^{3/2}u_2^{3/2}\\ &-u_1^{1/2}u_2^{1/2} - u_1^{-1/2}u_2^{-1/2} - u_1^{-3/2}u_2^{-3/2}.\\
P^L_\emptyset(u_1,u_2) = & u_1u_2^2 + u_1^2u_2 + u_1 + u_2 + \frac{1}{u_1} + \frac{1}{u_2} -u_1^2u_2^2 - u_1u_2 -1 - \frac{1}{u_1u_2}.\\
\end{align*}
$L = L_1 \sqcup L_2$ is a $2$ component link with both components unknots. The linking number of the $2$ components is $2$ so; 
\[P^L_{L_1}(u_2) = u_2\left(\sum_{i\geq 0}u_2^{-i}\right) \mbox{ and } P^L_{L_2}(u_1) = u_1\left(\sum_{i\geq 0}u_1^{-i}\right).\]
 
\end{exmp}

\begin{thm}\label{thm:alex} If $L$ is an $L$-space link, then for any \textnormal{\textbf{s}} $\in \mathbb{H}(L)$ and $r \in \{1,2,\ldots, l\}$:
\[\sum_{\substack{S \subset \{1,\ldots, l\}\\r \not\in S}}(-1)^{l-1-|S|} R_{\substack{\mbox{\textnormal{\scriptsize{\textbf{s}}}}'\geq \mbox{\textnormal{\scriptsize{\textbf{s}}}}\\s'_r = s_r}}(P^L_{L_S}) = 0 \textnormal{\mbox{ or }} 1.\]
\end{thm}
\begin{rem}\label{rem:l=1}
When $l = 1$, this says that the coefficients of $P^L_\emptyset$ are all $1$ or $0$, which follows from the work in \cite{OZLspace}.
\end{rem}
Given any $2$ variable polynomial $F(u_1,u_2)$, we define $F|_{(i,j)}$, where $i = 1$ or $2$, to be the polynomial obtained from $F$ by discarding all monomials where the exponent of $u_i$ is not equal to $j$. Then the above Theorem, when restricted to the $l = 2$ case, reads as follows:
\begin{cor}\label{cor:alex2}
Suppose that $L = L_1 \sqcup L_2$ is an $L$-space link. Then the nonzero coefficients of $P^L_\emptyset$ are all $\pm 1$. The nonzero coefficients of $P^L_{\emptyset}|_{(r,s'_r)}$ for $r = 1$ or $2$ and any $s'_r \in \mathbb{H}(L)_r$, alternate in sign. The first nonzero coefficient of $P^L_{\emptyset}|_{(r,s'_r)}$ is $-1$ if the coefficient of $u_r^{s'_r}$ in $P^L_{L_{3-r}}$ is $0$; and the first nonzero coefficient of $P^L_{\emptyset}|_{(r,s'_r)}$ is $1$ if the coefficient of $u_r^{s'_r}$ in $P^L_{L_{3-r}}$ is $1$. 
\end{cor}

\begin{proof}
As in Theorem \ref{thm:alex}, fix \textbf{s}$' = (s'_1,s'_2)$. Suppose without loss of generality that $r = 1$. We denote by $a_{s_1,s_2}$ the coefficient of $u_1^{s_1}u_2^{s_2}$ in $P^L_\emptyset(u_1,u_2)$, and $a_{s_1}$ the coefficient of $u_1^{s_1}$ in $P^L_{L_2}(u_1)$. Then according to Theorem \ref{thm:alex}:
\begin{equation}\label{eq:e1}
a_{s'_1} - \sum_{s_2\geq s'_2} a_{s'_1,s_2} = 0 \mbox{ or } 1.
\end{equation}
Similarly;
\begin{equation}\label{eq:e2}
a_{s'_1} - \sum_{s_2\geq s'_2+1} a_{s'_1,s_2} = 0 \mbox{ or } 1.
\end{equation}
Subtracting \ref{eq:e1} from \ref{eq:e2} gives $a_{s'_1,s'_2} = -1,0$ or $1$. We have thus shown that all the coefficients of $P^L_{\emptyset}$ or $-1, 0$ or $-1$. We know that $a_{s'_1}$ must be either $1$ or $0$ (see Remark \ref{rem:l=1}). Combining this with equation \ref{eq:e1} gives that  $\sum_{s_2\geq s'_2} a_{s'_1,s_2} = 0$ or $1$ if $a_{s'_1} = 1$, and $\sum_{s_2\geq s'_2} a_{s'_1,s_2} = 0$ or $-1$ if $a_{s'_1} = 0$. The rest of the corollary now immediately follows. 
\end{proof}

Part of the above corollary was already shown directly in Theorem $1.15$ of \cite{YLspace}. Additionally in \cite{YLspace}, it was shown that when $q$ and $k$ are odd positive integers $b(qk-1,-k)$ is an $L$-space link. This was conjectured to be a complete list of $2$-bridge $L$-space links, which is correct.
\begin{thm}\label{thm:bridge}
If $L$ is a $2$-bridge $L$-space link, then, after possibly reversing the orientation of one of the components, $L$ is equivalent to $b(qk - 1,-k)$ for some positive odd integers $q$ and $k$. 
\end{thm} 

The organization of this paper is as follows. Section $2$ consists of some homological algebra needed to compute $HFL^-(L)$ from its Euler characteristic when $L$ is a $2$ or $3$ component $L$-space link. Section $3$ generalizes the arguments in \cite{OZLinks} to work on links. In Section $4$ Theorem \ref{thm:main} is proved, as well the the restrictions on the Alexander polynomials of $L$-spaces. In Section $5$ we prove the classification of $2$ bridge $L$-space links.

\begin{ack}
First and foremost, I would like to thank my advisor Yi Ni for sharing his knowledge and providing invaluable support. Much of the work in this paper is directly inspired by Yajing Liu's paper \cite{YLspace}. Finally I would like to thank the Troesh family for partially funding me during this project through the Troesh family fellowship. 
\end{ack}

\section{Homological Preliminaries}\label{sec:Homological}
\begin{defn} Let $E_n = \{0,1,2\}^n \subset \mathbb{R}^n$ where $n\geq 1$. We will denote $(0,0,\ldots,0)$, $(1,1,\ldots 1)$ and $(2,2,\ldots,2)$ by $\boldsymbol{0}$,$\boldsymbol{1}$ and $\boldsymbol{2}$ respectively. For any $\varepsilon \in E_n$, we denote by $\varepsilon_j$ the $j$th coordinate of $\varepsilon$ and by $e_j$ the $j$th elementary coordinate vector. We define an \textbf{n-dimensional short exact cube of chain complexes},  $\boldsymbol{C}$ (or \textbf{short exact cube} for short), as follows:
\begin{description}
\item[1] For every $\varepsilon \in E_n$ there is a chain complex $\boldsymbol{C}_\varepsilon$ over $\mathbb{F}$. 
\item[2] Suppose that $\varepsilon', \varepsilon$ and $\varepsilon''$ are in $E_n$ and only differ in the $j$th coordinate with $\varepsilon'_j = 0, \varepsilon_j = 1$ and $\varepsilon''_j = 2$.
Then there is a short exact sequence
\[\xymatrix{
0 \ar[r] &\boldsymbol{C}_{\varepsilon'} \ar[r]^{i_{\varepsilon'\varepsilon}} &\boldsymbol{C}_{\varepsilon} \ar[r]^{j_{\varepsilon\varepsilon''}}& \boldsymbol{C}_{\varepsilon''} \ar[r]& 0.\\
}\]
\item[3] The diagram made by all of the complexes $\boldsymbol{C}_\varepsilon$ and maps $i_{\varepsilon'\varepsilon}, j_{\varepsilon\varepsilon''}$ is commutative.
\end{description}
We will denote  $\boldsymbol{C}_{(2,2,\ldots,2)}$ as $\overline{\boldsymbol{C}}$ for short. We define the $\textbf{cube of inclusions}, \boldsymbol{C}^I$, to be the sub-diagram consisting of all the chain complexes $\boldsymbol{C}_\varepsilon$ with $\varepsilon \in \{0,1\}^n$ and the corresponding inclusion maps. We call a short exact cube \textbf{basic} if the following additional properties hold:
\begin{description}
\item[4] For $\varepsilon \in \{0,1\}^n$, $H_*(\boldsymbol{C}_\varepsilon) \cong \mathbb{F}[U]$ where multiplication by $U$ drops homological grading by $2$. We do not specify what the top grading for $\mathbb{F}[U]$ is, but we do require that it is even.
\item[5] All of the maps $(i_{\varepsilon'\varepsilon})_*$, induced by homology in the cube of inclusions are either isomorphisms in all degrees, or $(i_{\varepsilon'\varepsilon})_*$ is injective in all degrees and the top degree supported in $H_*(C_\varepsilon)$ is $2$ higher than the top degree supported in $H_*(C_{\varepsilon'})$. Alternatively, $UH_*(\boldsymbol{C}_\varepsilon) \cong H_*(\boldsymbol{C}_{\varepsilon'})$
\end{description}
\end{defn}
When the top grading for $\mathbb{F}[U]$ is $d$, we will write it as $\mathbb{F}_{(d)}[U]$. Similarly, $\mathbb{F}_{(d)}$ will be used to denote $\mathbb{F}$ supported in degree $d$.

Given an $n$ dimensional basic short exact cube $\boldsymbol{C}$, if we restrict to the commutative diagram coming from the subset of $E_n$ with $j$th coordinate $i$ where $i = 0,1$ or $2$, this can be thought of as an $n-1$ dimensional short exact cube of chain complexes which we will denote by ${}^j_i\boldsymbol{C}$. For any $j\in \{1,2,\ldots,n\}$, $\overline{{}^j_2\boldsymbol{C}}$ is the same as $\overline{\boldsymbol{C}}$; and ${}^j_0\boldsymbol{C}$ and ${}^j_1\boldsymbol{C}$ are basic.

\begin{lem}
Suppose $\boldsymbol{C}$ is a basic short exact cube of chain complexes. Also let $\varepsilon \in E_n$ have some coordinate equal to $2$. Then, $H_*(\boldsymbol{C}_\varepsilon)$ is finite dimensional.
\end{lem}
\begin{proof}
In the $n=1$ case, $H_*(\overline{\boldsymbol{C}})$ is either $\mathbb{F}$ or $0$ by property $5$ of basic short exact cubes. Thus, for any $n$-dimensional basic short exact cube $\boldsymbol{C}$, the homologies of the complexes in ${}^{j_1}_2\boldsymbol{C}^I$ are only either $\mathbb{F}$ or $0$ for any $j_1$. From here we can conclude that the homologies of the complexes in ${}^{j_2}_2{}^{j_1}_{2}\boldsymbol{C}^I$ are finite and continuing with this argument proves the claim.
\end{proof}

\begin{defn}\label{def:hc}
If $\boldsymbol{C}$ is a basic short exact cube, then we define the \textbf{hypercube graph of $\boldsymbol{C}$}, $HC(\boldsymbol{C})$, as a directed graph with labeled edges as follows:
\begin{itemize}
\item The vertices correspond to the elements of the set $\{0,1\}^n$.
\item There is a directed edge from $\varepsilon'$ to $\varepsilon$ if the two agree in all coordinates except the $j$th for some $1\leq j \leq n$ and $\varepsilon'_j = 0, \varepsilon_j = 1$. We will denote the edge from $\varepsilon'$ to $\varepsilon$ by $e_{\varepsilon'\varepsilon}$.
\item An edge $e_{\varepsilon'\varepsilon}$ is labeled with $0$ if $(i_{\varepsilon'\varepsilon})_*$ is an isomorphism in all degrees and $1$ otherwise. We will denote the label of an edge $e$ by $l_{\boldsymbol{C}}(e_{\varepsilon'\varepsilon})$ or $l(e_{\varepsilon'\varepsilon})$ when $\boldsymbol{C}$ is clear from context.
\end{itemize}
We will denote by $\widetilde{HC}(\boldsymbol{C})$ the subgraph of $HC(\boldsymbol{C})$ induced by all the vertices except the origin and we will refer to $\widetilde{HC}(\boldsymbol{C})$ as the \textbf{hypercube subgraph of $\boldsymbol{C}$}.
\end{defn}

\begin{rem}\label{rem:path}
Note that, since $\boldsymbol{C}^I$ is a commutative diagram, for any two directed paths between vertices the sum of the edge labels must be the same in $HC(\boldsymbol{C})$. If we are given a directed hypercube graph $G$ (directed as in definition \ref{def:hc}) with edge labels $0$ and $1$ that satisfies the property that the sum of the edge labels along any two directed paths between vertices is the same, we can easily construct a basic short exact cube with $G$ as its hypercube graph. Also note that $\chi(H_*(\overline{\boldsymbol{C}}))$ is completely determined by $HC(\boldsymbol{C})$.
\end{rem}

\begin{lem}\label{lem:HC}
Suppose that $\boldsymbol{C}$ is a basic short exact cube. There are only two mutually exclusive possibilities:
\begin{description}
\item[1] If $\boldsymbol{C}'$ is another basic short exact cube then $\widetilde{HC}(\boldsymbol{C'}) = \widetilde{HC}(\boldsymbol{C}) \Rightarrow HC(\boldsymbol{C}') = HC(\boldsymbol{C})$.
\item[2] Either all of the edges in $HC(\boldsymbol{C})\char92 \widetilde{HC}(\boldsymbol{C})$ (i.e. all the edges emerging from $\boldsymbol{0}$ )are labeled with $0$ or they are all labeled with $1$.
\end{description}
\end{lem}
\begin{proof}
Note first that, if possibility $1$ is satisfied, possibility $2$ cannot also be satisfied since if all the edges emerging from $\boldsymbol{0}$ are labeled with $i$ (where $i$ is $0$ or $1$) then we can get another valid labeling by simply replacing all the $i$'s emerging from the origin with $(1-i)$s (see Remark \ref{rem:path}).\\\\
Suppose that $\boldsymbol{C}$ and $\boldsymbol{C}'$ satisfy $\widetilde{HC}(\boldsymbol{C'}) = \widetilde{HC}(\boldsymbol{C})$, but $HC(\boldsymbol{C}') \neq HC(\boldsymbol{C})$. Then there must be some $\varepsilon'$ connected to the origin such that the edge from $\boldsymbol{0}$ to $\varepsilon'$ is labeled differently in $HC(\boldsymbol{C}')$ and $HC(\boldsymbol{C})$. Assume without loss of generality that $l_{\boldsymbol{C}}(e_{\boldsymbol{0}\varepsilon'}) = 1$ and $l_{\boldsymbol{C}'}(e_{\boldsymbol{0}\varepsilon'}) = 0$. Consider any other vertex $\varepsilon$ connected to the origin and consider $l_{\boldsymbol{C}}(e_{\boldsymbol{0}\varepsilon})$. We claim that $l_{\boldsymbol{C}}(e_{\boldsymbol{0}\varepsilon})$ must be $1$. To see this, consider the square subgraph induced by the vertices $\boldsymbol{0}, \varepsilon,\varepsilon'$ and $\delta = \varepsilon + \varepsilon'$. If $l_{\boldsymbol{C}}(e_{\boldsymbol{0}\varepsilon}) = 0$ then since $l_{\boldsymbol{C}}(e_{\boldsymbol{0}\varepsilon'}) = 1$ this forces $l_{\boldsymbol{C}}(e_{\varepsilon\delta}) = 1 = l_{\boldsymbol{C}'}(e_{\varepsilon\delta})$ and $l_{\boldsymbol{C}}(e_{\varepsilon'\delta}) =0 = l_{\boldsymbol{C}'}(e_{\varepsilon'\delta})$ (see the Remark \ref{rem:path}). However this is impossible because we know $l_{\boldsymbol{C}'}(e_{\boldsymbol{0}\varepsilon'}) = 0$ and if $0 = l_{\boldsymbol{C}'}(e_{\varepsilon'\delta}), 1 = l_{\boldsymbol{C}'}(e_{\varepsilon\delta})$ there is no label that works for $e_{\boldsymbol{0}\varepsilon'}$. So we get that in $\boldsymbol{C}$ every edge emerging from the origin must be labeled $1$ if one of them is. By the same argument, we can show that every edge emerging from the origin must be labeled $0$ if one of them is. This proves that the two cases stated in the Lemma are exhaustive and mutually exclusive.\\
\end{proof}

\begin{lem}\label{lem:echar}
Suppose that $\boldsymbol{A}$ and $\boldsymbol{B}$ are two basic short exact cubes satisfying $\widetilde{HC}(\boldsymbol{A}) = \widetilde{HC}(\boldsymbol{B})$, every edge in $HC(\boldsymbol{A})\char92 \widetilde{HC}(\boldsymbol{A})$ is labeled with $0$, and every edge in $HC(\boldsymbol{B})\char92 \widetilde{HC}(\boldsymbol{B})$ is labeled with $1$. Then,
\[\chi(H_*(\overline{\boldsymbol{A}})) = \chi(H_*(\overline{\boldsymbol{B}})) + (-1)^n.\]
\end{lem}
\begin{proof}
We will prove this inductively. For the $n=1$ case using the fact that both $H_*(\boldsymbol{A}_0)$ and $H_*(\boldsymbol{B}_0)$ have even top grading we directly compute that $\chi(H_*(\overline{\boldsymbol{A}})) = 0$ and $\chi(H_*(\overline{\boldsymbol{B}})) = 1$.
Now we can proceed with the induction. Note that we have:
\[\chi(\overline{\boldsymbol{A}}) = \chi(\overline{{}^1_1\boldsymbol{A}}) - \chi(\overline{{}^1_0\boldsymbol{A}})\; \mbox{and}\; \chi(\overline{\boldsymbol{B}}) = \chi(\overline{{}^1_1\boldsymbol{B}}) - \chi(\overline{{}^1_0\boldsymbol{B}}).\] 
$\chi(\overline{{}^1_1\boldsymbol{A}}) = \chi(\overline{{}^1_1\boldsymbol{B}})$ since they are both completely determined by the hypercube subgraph $\widetilde{HC}$ and also $\chi(\overline{{}^1_0\boldsymbol{A}}) = \chi(\overline{{}^1_0\boldsymbol{B}}) + (-1)^{n-1}$ by induction.
\end{proof}

\begin{lem}\label{lem:comp}
Suppose that $\boldsymbol{C}$ is a $1,2$ or $3$-dimensional basic cube of chain complexes. Then we can compute $H_*(\overline{\boldsymbol{C}})$ as a graded vector space if we know $H_*(\boldsymbol{C}_\varepsilon)$ for any $\boldsymbol{C}_\varepsilon$ in the cube of inclusions $\boldsymbol{C}^I$, as well as all the maps $(i_{\varepsilon'\varepsilon})_*$ induced by homology in the cube of inclusions $\boldsymbol{C}^I$.
\end{lem}
\begin{proof}
 When $n = 1$, we have a short exact sequence:
\[\xymatrix{
0 \ar[r] &\boldsymbol{C}_{0} \ar[r]^{i_{01}} &\boldsymbol{C}_{1} \ar[r]^{j_{12}}& \overline{\boldsymbol{C}} \ar[r]& 0.\\
}\]
Thus if $(i_{01})_*$ is an isomorphism, we get that $H_*(\overline{\boldsymbol{C}}) \cong 0$; and if not, then $H_*(\overline{\boldsymbol{C}}) \cong \mathbb{F}$.
For the $n=2$ case we show all possibilities for $HC$ in Figure \ref{fig:lem2}. 
\begin{figure}[h]
    \centering
		\includegraphics[scale = .8]{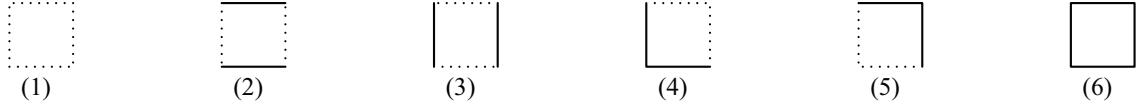}
    \caption{All possible hypercube graphs in the $n=2$ case ((0,0) is on the bottom left). The dotted lines denote edges labeled with $0$ and the solid lines are edges labeled with $1$}
		\label{fig:lem2}
\end{figure}

If we assume that $H_*(\boldsymbol{C}_{\boldsymbol{0}}) \cong \mathbb{F}_{(0)}[U]$, then $H_*(\overline{\boldsymbol{C}})$ is $0,0,0,\mathbb{F}_{(3)},\mathbb{F}_{(2)},\mathbb{F}_{(4)}\oplus\mathbb{F}_{(3)}$ for the $6$ possibilities shown in Figure \ref{fig:lem2}, respectively.\\
In the $n=3$ case we only need to consider those $HC$ which do not have a facet equal to $(1), (2)$ or $(3)$ in Figure \ref{fig:lem2} , as otherwise we would have for some $j = 1,2$ or $3$, $H_*(\overline{{}^j_0\boldsymbol{C}}) = 0$ or $H_*(\overline{{}^j_1\boldsymbol{C}}) = 0$. This would allow us to compute $H_*(\overline{\boldsymbol{C}})$ from the long exact sequence for the short exact sequence: 
\[\xymatrix{
0 \ar[r] &\overline{{}^j_0\boldsymbol{C}} \ar[r] &\overline{{}^j_1\boldsymbol{C}} \ar[r] & \overline{\boldsymbol{C}} \ar[r]& 0.\\
}\]
We show all the possibilities for $HC$ when $n=3$ and none of the facets are as $(1), (2)$ or $(3)$ of Figure \ref{fig:lem2} in Figure \ref{fig:lem3}.

\begin{figure}[h]
    \centering
		\includegraphics{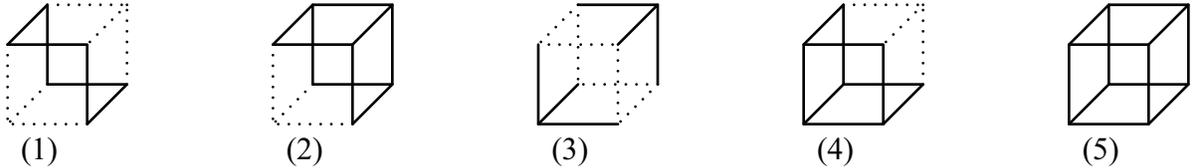}
    \caption{Some hypercube graphs in the $n=3$ case. Once again the $(0,0,0)$ is on the bottom left and the dotted lines denote edges labeled with $0$ and the solid lines are edges labeled with $1$}
		\label{fig:lem3}
\end{figure}
If we assume that $H_*(\boldsymbol{C}_{\boldsymbol{0}}) \cong \mathbb{F}_{(0)}[U]$, then $H_*(\overline{\boldsymbol{C}})$ is $\mathbb{F}_{(3)}^2,\mathbb{F}_{(4)}\oplus \mathbb{F}_{(3)}^2, \mathbb{F}_{(4)}^2, \mathbb{F}_{(5)}^2\oplus\mathbb{F}_{(4)}, \mathbb{F}_{(6)}\oplus\mathbb{F}_{(5)}^2\oplus\mathbb{F}_{(4)}$ for the five cases shown, respectively.
\end{proof}

\begin{rem}\label{rem:4fail}
The above Lemma does not hold when $n \geq 4$. Consider the basic $4$ dimensional short exact cube $\boldsymbol{C}$ where every edge of $HC(\boldsymbol{C})$ is labeled with $1$ and $H_*(\boldsymbol{C}_{\boldsymbol{0}}) \cong \mathbb{F}_{(0)}[U]$. For any $j_1, j_2 \in \{1,2,3,4\}$ we have $H_*({}^{j_1}_2{}^{j_2}_2\boldsymbol{C}_{00}) \cong \mathbb{F}_4\oplus \mathbb{F}_3, H_*({}^{j_1}_2{}^{j_2}_2\boldsymbol{C}_{10}) \cong H_*({}^{j_1}_2{}^{j_2}_2\boldsymbol{C}_{01}) \cong \mathbb{F}_6\oplus \mathbb{F}_5$ and $H_*({}^{j_1}_2{}^{j_2}_2\boldsymbol{C}_{11}) \cong \mathbb{F}_8\oplus \mathbb{F}_7$. It follows that all maps on homology in the cube of inclusions for ${}^{j_1}_2{}^{j_2}_2\boldsymbol{C}$ are trivial. So for all $j$ the map from $H_*(\overline{{}^{j}_0\boldsymbol{C}}) \cong \mathbb{F}_{(6)}\oplus\mathbb{F}_{(5)}^2\oplus\mathbb{F}_{(4)}$ to $H_*(\overline{{}^{j}_1\boldsymbol{C}})\cong \mathbb{F}_{(8)}\oplus\mathbb{F}_{(7)}^2\oplus\mathbb{F}_{(6)}$ may be of rank $0$ or $1$ without violating commutativity. Thus $H_*(\overline{C})$ may be either $\mathbb{F}_{(8)}\oplus\mathbb{F}_{(7)}^3\oplus\mathbb{F}_{(6)}^3\oplus \mathbb{F}_{(5)}$ or $\mathbb{F}_{(8)}\oplus\mathbb{F}_{(7)}^2\oplus\mathbb{F}_{(6)}^2\oplus \mathbb{F}_{(5)}$. See also Theorem $1.5.1.d$ in \cite{GNLspace}.
\end{rem}

\section{The Chain Complex}\label{sec:Chain}

For a complete overview of Heegaard Floer homology, admissible multi-pointed Heegaard diagrams for knots and links, the definition of $L$-spaces and their relationship with the Heegaard Floer complex, see \cite{OZOrig},\cite{OZApps}, \cite{OZKnot}, \cite{OZLinks}, \cite{OZLspace}, \cite{OZIntsurg}, \cite{OZTnorm} and  \cite{MOComb}. Suppose that $L \subset S^3$ is an oriented $l$ component link. In this paper, we define a multi-pointed Heegaard diagram $\mathcal{H} = (\Sigma_g, \boldsymbol{\alpha},\boldsymbol{\beta},$\textbf{w},\textbf{z}$)$ for $L$ with the following properties\footnote{ This is identical to the definition given in \cite{OZLinks} except we want to allow ``spare'' basepoints that will arise in the proof of the main theorem.}:
\begin{itemize}
\item $\Sigma_g$ is a closed oriented surface of genus $g$.
\item $\boldsymbol{\alpha} = (\alpha_1, \ldots , \alpha_{g+m-1})$ is a collection of disjoint simple closed curves which span a $g$-dimensional lattice of $H_1(\Sigma, \mathbb{Z})$, and the same goes for $\beta = (\beta_1, \ldots, \beta_{g+m-1})$. Thus, $\boldsymbol{\alpha}$ and $\boldsymbol{\beta}$ specify handlebodies $U_\alpha$ and $U_\beta$. We require that $U_\alpha \cup_\Sigma U_\beta = S^3$.
\item \textbf{z} $= (z_1,z_2,\ldots, z_l)$ and \textbf{w} $ = (w_1,w_2,\ldots, w_m)$ are both collections of basepoints in $\Sigma$ where $l \leq m$. We will call $w_{l+1}, w_{l+2}, \ldots w_m$ free basepoints.
\item If $\{A_i\}_{i=1}^m$ and $\{B_i\}_{i=1}^m$ are the connected components of $\Sigma\char92\left(\bigcup_{i=1}^{g+m-1} \alpha _i\right)$ and $\Sigma\char92\left(\bigcup_{i=1}^{g+m-1} \beta _i\right)$, respectively then $w_i \in A_i \cap B_i$ for any $1 \leq i \leq m$; and there is some permutation $\sigma$ of $\{1,\ldots, l\}$ such that $z_i \in A_i \cap B_{\sigma(i)}$ when $1 \leq i \leq l$.
\item The diagram as defined so far specifies the link $L \subset S^3$.
\item We require that all of the $\alpha$ and $\beta$ curves intersect transversely and that every non-trivial periodic domain have both positive and negative local multiplicities (see section $3.4$ of \cite{OZLinks}). 
\end{itemize}

Also recall that for every intersection point \textbf{x}$ \in \mathbb{T}_\alpha \cap \mathbb{T}_\beta$ there is a Maslov grading $M($\textbf{x}$)$ and an Alexander multigrading $A_i($\textbf{x}$) \in \mathbb{H}(L)_i$.

\begin{defn} Suppose we have a multi-pointed Heegaard diagram $\mathcal{H} = (\Sigma_g, \boldsymbol{\alpha},\boldsymbol{\beta},$\textbf{w},\textbf{z}$)$ for the pair $L$ as above. We define the complex $CF^-(\mathcal{H})$ to be free over $\mathbb{F}$ with generators $[$\textbf{x}$,i_1,j_1,\ldots, i_l$ $,j_l, i_{l+1},\ldots, i_m]$ where $i_k \in \mathbb{Z}_{\leq 0}$, and  $j_k \in \mathbb{Q}$ satisfying $j_k - i_k = A_k($\textbf{x}$)$. The differential is, as usual, given by counting holomorphic disks:
\[\partial[\mbox{\textbf{x}},i_1,j_1,\ldots, i_l,j_l, i_{l+1}, \ldots, i_m] =\]
\[ \sum_{\mbox{\scriptsize{\textbf{y}}} \in \mathbb{T}_\alpha \cap \mathbb{T}_\beta}\sum_{\substack{\phi\in\pi_2(\mbox{\scriptsize{\textbf{x},\textbf{y}}})\\ \mu(\phi) = 1}} [\mbox{\textbf{y}}, i_1-n_{w_1}(\phi), j_1 - n_{z_1}(\phi),\ldots , i_l-n_{w_l}(\phi), j_l - n_{z_l}(\phi),i_{l+1} - n_{w_{l+1}}(\phi), \ldots, i_m - n_{w_m}(\phi)].\]
In the notation of \cite{MOSurg}, this differential and Heegaard diagram correspond to the maximally colored case.
\end{defn}
The complex $CF^-$ is also an $\mathbb{F}[U_1,U_2,\ldots, U_m]$-module. The action of $U_k$ for $1\leq k \leq l$ is given by:
\[U_k[\mbox{\textbf{x}},i_1,j_1,\ldots,i_k,j_k,\ldots i_l,j_l,i_{l+1}\ldots,i_m] = [\mbox{\textbf{x}},i_1,j_1,\ldots,i_k-1,j_k-1,\ldots i_l,j_l,i_{l+1},\ldots,i_m];\]
and for $l<k<m$ is given by;
\[U_k[\mbox{\textbf{x}},i_1,j_1,\ldots, i_l,j_l, i_{l+1}, \ldots, i_k, \ldots, i_m] = [\mbox{\textbf{x}},i_1,j_1,\ldots, i_l,j_l,i_{l+1} \ldots, i_{k}-1, \ldots, i_m].\]
We define the Maslov grading of $[\mbox{\textbf{x}},i_1,j_1,\ldots,i_k,j_k,\ldots i_l,j_l,i_{l+1}\ldots,i_m]$ by setting it equal to $M($\textbf{x}$)$ when all the $i_k$ are $0$ and letting the action of each $U_i$ drop the maslov grading by $2$. 
Note that both as a complex and $\mathbb{F}[U_1,U_2,\ldots, U_m]$-module $CF^-$ is isomorphic to $CF^-$ as defined in \cite{OZLinks} via the isomorphism induced by
\[[\mbox{\textbf{x}},i_1,j_1,\ldots, i_l,j_l, i_{l+1}, \ldots, i_m]\mapsto U_1^{-i_1}\ldots U_l^{-i_m}\mbox{\textbf{x}}.\]
And so it follows that $CF^-$ is a chain complex with homology $HF^-(S^3)$. 
\begin{defn} Suppose that we have a Heegaard diagram $\mathcal{H}$ for $L \subset S^3$ as above. Fix some \textbf{s} $= (s_1,\ldots, s_l) \in \mathbb{H}(L)$. Now suppose that we restrict $CF^-(\mathcal{H})$ to only those generators $[\mbox{\textbf{x}},i_1,j_1,\ldots, i_l$ $,j_l, i_{l+1}, \ldots, i_m]$ which satisfy $A_k($\textbf{x}$) = j_k$ and force the differential to only count holomorphic disks $\phi$ with $n_{z_k}(\phi) = 0$ when $1\leq k \leq l$. Then this quotient complex of $CF^-(\mathcal{H})$ will be denoted by $CFL^-(\mathcal{H}$,\textbf{s}$)$. Note that $CFL^-$ inherits an $\mathbb{F}[U_1,\ldots U_m]$ module action from $CF^-$.
\end{defn}
\begin{thm} If $CFL^-(\mathcal{H}$, \textnormal{\textbf{s}}$)$ is as above, then its homology is $HFL^-(S^3,L$, \textnormal{\textbf{s}}$)$
\end{thm}
\begin{proof}
(This is very similar to Proposition $5.8$ in \cite{YLspace}.) If the diagram $\mathcal{H}$ has no free points, then $CFL^-(\mathcal{H}$,\textbf{s}$)$ is the same as the complex computing $HFL^-$ in \cite{OZLinks}. so we only need to show what happens in the case when there are free basepoints in $\mathcal{H}$. Suppose that $\mathcal{H}'$ is another Heegaard diagram that only has $l$-pairs of basepoints (one pair for each link component) and no others. Then we claim that $\mathcal{H}$ can be obtained from $\mathcal{H}'$ via the following moves:
\begin{description}
\item[1] a 3-manifold isotopy 
\item[2] $\alpha$ and $\beta$ curve isotopy
\item[3] $\alpha$ and $\beta$ handleslide
\item[4] index one/two stabilization
\end{description}
We may also need the inverses of moves $1$-$4$
\begin{description}
\item[5] free index zero/three stabilization, 
\end{description}
but we do not need the inverse of $5$.

We follow the argument from proposition 4.13 of \cite{MOSurg} which relies on \cite{MOComb} Lemma $2.4$ . Basically, we can apply moves $1$-$4$ to $\mathcal{H'}$  to obtain a Heegaard diagram that differs from a diagram with exactly $l$ pairs of basepoints (one pair for each component) by index zero/three stabalizations only. Then we can apply moves $1-4$ again to obtain the diagram $\mathcal{H}$. Now we know that moves $1-4$ and their inverses give chain homotopy equivalences for the complexes $CFL^-$ by the arguments given in \cite{OZOrig} and proposition $3.9$ of \cite{OZLinks}, so we will focus on move $5$. Suppose that $\mathcal{H}_1$ and $\mathcal{H}_2$ are two Heegaard diagrams for $L$, and $\mathcal{H}_2$ is obtained from $\mathcal{H}_1$ by a free index zero/three stabilization. Then $\mathcal{H}_2$ has an extra free basepoint $w_r$ that $\mathcal{H}_1$ does not have. By the argument of Lemma $6.1$ in \cite{OZLinks}, we see that the complex $CFL^-(\mathcal{H}_2$, \textbf{s}$)$ is just the mapping cone
\[\xymatrix{CFL^-(\mathcal{H}_1,\mbox{ \textbf{s}})[U_r] \ar[r]^{U_r - U_k} & CFL^-(\mathcal{H}_1,\mbox{ \textbf{s}})[U_r],}\]
where $k$ is an index corresponding to some \textbf{w} basepoint in $\mathcal{H}_1$. Now, $k$ may correspond to a free basepoint, or it may correspond to some link component (in which case the action of $U_k$ is trivial); but in either case, the homology of this mapping cone is the same as the homology of $CFL^-(\mathcal{H}_1$, \textbf{s}$)$. So we see that all of the above $5$ Heegaard moves induce quasi-isomorphisms of chain complexes, and this gives the desired result. 
\end{proof}

\begin{defn}
Fix a Heegaard diagram $\mathcal{H}$ for $L$. For a given \textbf{s} $\in \mathbb{H}(L)$ and $\varepsilon \in E_l$, we define the complex $A_{\mbox{\scriptsize{\textbf{s}}},\varepsilon}^-(\mathcal{H})$ to be the quotient complex of $CF^-(\mathcal{H})$ generated by those $[\mbox{\textbf{x}},i_1,j_1,\ldots, i_l,j_l, i_{l+1},$ $ \ldots, i_m]$ that satisfy
\begin{itemize}
\item $\mbox{max}\{i_k,j_k-(s_k-1)\} \leq 0$ if $\varepsilon_k = 0$
\item $\mbox{max}\{i_k,j_k-s_k\} \leq 0$ if $\varepsilon_k = 1$
\item $i_k \leq 0$ and $j_k = s_k$ if $\varepsilon_k = 2$.
\end{itemize}
By $A_{\scriptsize{\mbox{\textbf{s}}}}^-(\mathcal{H})$ we mean $A_{\scriptsize{\mbox{\textnormal{\textbf{s}}}},\scriptsize{\mbox{\textbf{1}}}}^-(\mathcal{H})$. We will write $A_{\scriptsize{\mbox{\textbf{s}}},\varepsilon}^-$ when the choice of diagram is clear from context. The complex $A_{\scriptsize{\mbox{\textnormal{\textbf{s}}}},\scriptsize{\mbox{\textbf{1}}}}^-(\mathcal{H})$ inherits an $\mathbb{F}[U_1, \ldots U_m]$ action from $CF^-(\mathcal{H})$. When $\mathcal{H}$ is clear from context we will omit $\mathcal{H}$ from the notation.
\end{defn}
\begin{rem}\label{rem:diff}
If we complete $A_{\scriptsize{\mbox{\textbf{s}}}}^-(\mathcal{H})$ with respect to the maximal ideal $(U_1,\ldots, U_m)$, there is an isomorphism between the completed version of $A_{\scriptsize{\mbox{\textbf{s}}}}^-(\mathcal{H})$ and  $\mathfrak{A}^-(\mathcal{H}, \mbox{\textbf{s}})$ as defined in section $4.2$ of \cite{MOSurg}, given by:
\[[\mbox{\textbf{x}},i_1,j_1,\ldots, i_l,j_l,i_{l+1},\ldots, i_m]\mapsto U_1^{-\mbox{\scriptsize{max}}\{i_1,j_1-s_1\}}U_2^{-\mbox{\scriptsize{max}}\{i_2,j_2-s_2\}}\ldots U_l^{-\mbox{\scriptsize{max}}\{i_l,j_l-s_l\}}U^{-i_{l+1}}\ldots U^{-i_m}\mbox{\textbf{x}}.\]
We can use the proofs in section $4.3$ and $4.4$ of \cite{MOSurg} to show that the homology of the complex $A_{\scriptsize{\mbox{\textbf{s}}}}^-(\mathcal{H})$ does not depend on the choice of a Heegaard diagram. For this reason we will sometimes write $H_*(A_{\scriptsize{\mbox{\textbf{s}}}}^-(\mathcal{H}))$ as $H_*(A_{\scriptsize{\mbox{\textbf{s}}}}^-(L))$. In this paper we could have just used the complexes $\mathfrak{A}^-_{\scriptsize{\mbox{\textbf{s}}}}$ to get the same results about link Floer homology. The choice to use the notation here has been made to make the analogy with the work in \cite{OZKnot} and \cite{OZLspace} more clear.
\end{rem}

\begin{thm} \label{thm:Lprop}Suppose that $L \subset S^3$ is an $L$-space link and \textbf{s} $\in \mathbb{H}(L)$. Then, as $\mathbb{F}[U_1,U_2,\ldots,U_l]$-modules,
\[H_*(A_{\scriptsize{\mbox{\textnormal{\textbf{s}}}}}^-) \cong \mathbb{F}[U].\]
where all of the $U_i$ have the same action as $U$ on the right hand side.
\end{thm}
\begin{proof}
We can use the proof of Theorem\footnote{ As was mentioned in Remark \ref{rem:diff}, the only difference between the complex in that paper and this one is that it is defined over $\mathbb{F}[[U_1,U_2,\ldots,U_m]]$ as opposed to $\mathbb{F}[U_1,U_2,\ldots,U_m]$. However the proof of Theorem $10.1$ in \cite{MOSurg} does not rely on $\mathbb{F}[[U_1,U_2,\ldots,U_m]]$ in any way. See also the proof of Theorem $4.1$ in \cite{OZKnot}.} $10.1$ in \cite{MOSurg} to see that for any \textbf{s} $\in \mathbb{H}(L)$, $H_*(A_{\scriptsize{\mbox{\textnormal{\textbf{s}}}}}^-)$ is isomorphic (as a module) to $HF^-(Y,\mathfrak{s})$, where $Y$ is some $L$ space obtained by large positive surgery on $L$ and $\mathfrak{s}$ is a Spin$^c$ structure over $Y$. 
\end{proof}

\begin{rem}
The above property characterizes $L$-space links. See also proposition $1.11$ of \cite{YLspace}.
\end{rem}

Suppose that, for a fixed \textbf{s} $\in \mathbb{H}(L)$, we have $\varepsilon', \varepsilon$ and $\varepsilon''$ in $E_l$ so that they only differ in the $j$th coordinate with $\varepsilon'_j = 0, \varepsilon_j = 1$ and $\varepsilon''_j = 2$. Then, for a given Heegaard diagram $\mathcal{H}$ of $L$, there is a short exact sequence:
\[\xymatrix{
0 \ar[r] &A_{\scriptsize{\mbox{\textnormal{\textbf{s}}}},\varepsilon'}^-(\mathcal{H}) \ar[r]^{i_{\varepsilon'\varepsilon}} &A_{\scriptsize{\mbox{\textnormal{\textbf{s}}}},\varepsilon}^-(\mathcal{H}) \ar[r]^{j_{\varepsilon\varepsilon''}}& A_{\scriptsize{\mbox{\textnormal{\textbf{s}}}},\varepsilon''}^-(\mathcal{H}) \ar[r]& 0.\\
}\]
So, we can define a short exact cube of chain complexes $\boldsymbol{A}^-(\mathcal{H},\mbox{\textbf{s}})$ by setting $\boldsymbol{A}^-(\mathcal{H},\mbox{\textbf{s}})_\varepsilon = A_{\mbox{\scriptsize{\textnormal{\textbf{s}}}},\varepsilon}^-(\mathcal{H})$. Note also that $\overline{\boldsymbol{A}^-(\mathcal{H},\mbox{\textbf{s}})}$ is just $CFL^-(\mathcal{H},\mbox{\textbf{s}})$.

\begin{thm}\label{thm:Aminus} For any \textnormal{\textbf{s}} $\in \mathbb{H}(L)$, $\boldsymbol{A}^-(\mathcal{H},\mbox{\textnormal{\textbf{s}}})$ is a basic short exact cube when $L$ is an $L$-space link.
\end{thm}

\begin{proof}
We want to show properties $4$ and $5$ in definition $2.1$. Note that, by Theorem \ref{thm:Lprop}, we already know that for all $\varepsilon \in \{0,1\}^n$ we have $H_*(A_{\scriptsize{\mbox{\textnormal{\textbf{s}}}},\varepsilon}^-) \cong \mathbb{F}[U]$. First, we will examine all maps induced on homology in the cube of inclusions. Suppose that $\varepsilon'$ and $\varepsilon$ are in $\{0,1\}^l$ and differ only in the $j$th coordinate with $\varepsilon'_j = 0$ and $\varepsilon_j = 1$. Also define $\varepsilon''$ to agree in all coordinates with $\varepsilon$ except the $j$th and $\varepsilon''_j = 2$. Now, following the proof of Lemma $3.1$ in \cite{OZLspace}, we define $X$ to be the set of generators $[\mbox{\textbf{x}},i_1,j_1,\ldots, i_l,j_l,i_{l+1},\ldots,l_m]$ of $CF^-$ that satisfy:
\begin{description}
\item[1] $\mbox{max}\{i_k,j_k-(s_k-1)\} \leq 0$ if $\varepsilon''_k = 0$
\item[2] $\mbox{max}\{i_k,j_k-s_k\} \leq 0$ if $\varepsilon''_k = 1$
\item[3] $i_k \leq 0$ and $j_k = s_k$ if $\varepsilon''_k = 2$, i.e. when $k = j$.
\end{description}
We define a set $Y$ similarly, except \textbf{3} is replaced with;
\begin{description}
\item[3] $i_k = 0$ and $j_k < s_k$ if $\varepsilon''_k = 2$, i.e. when $k = j$.
\end{description}
Note that $X$ naturally generates a sub-complex of a quotient complex of $CF^-$, which we will denote by $C\{X\} = A_{\scriptsize{\mbox{\textbf{s}}},\varepsilon''}^-$. Similarly, there are complexes $C\{U_jX\}, C\{Y\}$, $C\{X \cup Y\}$, $C\{U_jX\cup Y\}$ and $C\{X \cup U_jX \cup Y\}$, all of which inherit differentials from $CF^-$. Since $C\{X \cup Y\} = A_{\scriptsize{\mbox{\textbf{s}}},\varepsilon}^-/U_j(A_{\scriptsize{\mbox{\textbf{s}}},\varepsilon}^-)$ its homology is $\widehat{HF}$ of some $L$-space obtained by some large surgery on $L$ (see section $11.2$ of \cite{MOSurg}). Therefore $H_*(C\{X \cup Y\}) \cong \mathbb{F}$. Similarly $H_*(C\{U_jX \cup Y\}) \cong \mathbb{F}$. Now we have two short exact sequences of complexes:
\[\xymatrix{0 \ar[r] &C\{Y\} \ar[r]^-{i_1} & C\{X \cup Y\} \ar[r]^-{j_1} & C\{X\}\ar[r] & 0 }\]
and
\[\xymatrix{0 \ar[r] &C\{U_jX\} \ar[r]^-{i_2} & C\{U_j X \cup Y\} \ar[r]^-{j_2} & C\{Y\}\ar[r] & 0. }\]
We will denote the connecting homomorphims for these two complexes by $\delta_1$ and $\delta_2$, respectively. First note that $\delta_2\circ\delta_1 = 0$ (this follows from the fact the differential $\partial$ on the quotient complex $C\{X \cup U_jX \cup Y\}$ satisfies $\partial^2 = 0$). Now it follows from the exact same argument as in Lemma $3.1$ in  \cite{OZLspace} that either $H_*(C\{X\}) = H_*(A_{\scriptsize{\mbox{\textbf{s}}},\varepsilon''}^-)$ is $0$ and $H_*(C\{Y\})$ is $\mathbb{F}$, or $H_*(C\{X\})$ is $\mathbb{F}$ and $H_*(C\{Y\})$ is $0$ . If $H_*(C\{X\}) = 0$ then the map $i_{\varepsilon'\varepsilon}: A_{\scriptsize{\mbox{\textnormal{\textbf{s}}}},\varepsilon'}^-  \to A_{\scriptsize{\mbox{\textnormal{\textbf{s}}}},\varepsilon}^-$ clearly induces an isomorphism on homology. If $H_*(A_{\scriptsize{\mbox{\textbf{s}}},\varepsilon''}^-)$ is $\mathbb{F}$ supported in some degree $k$ then it follows from the first short exact sequence that $H_*(C\{X \cup Y\}) = H_*(A_{\scriptsize{\mbox{\textbf{s}}},\varepsilon}^-/U_j(A_{\scriptsize{\mbox{\textbf{s}}},\varepsilon}^-))$ is also $\mathbb{F}$ supported in degree $k$. Then, from the second short exact sequence it follows that $H_*(C\{U_jX\}) \cong H_*(C\{U_jX \cup Y\}) = H_*(A_{\scriptsize{\mbox{\textbf{s}}},\varepsilon'}^-/U_j(A_{\scriptsize{\mbox{\textbf{s}}},\varepsilon'}^-))$ is $\mathbb{F}$ supported in degree $k-2$. So we now have that the top grading in $H_*(A_{\scriptsize{\mbox{\textbf{s}}},\varepsilon'}^-)$ is two less than the top grading in $H_*(A_{\scriptsize{\mbox{\textbf{s}}},\varepsilon}^-)$, and we have now completely verified property $5$ in the definition of a basic short exact cube.\\

The only thing that is left to check in property $4$ is that for any $\varepsilon \in \{0,1\}^l$,  $H_*(\boldsymbol{A}^-(\mathcal{H},\mbox{\textbf{s}})_\varepsilon) \cong \mathbb{F}[U]$ has even top degree. For any sufficiently large $(s_1, s_2, \ldots, s_l) =$ \textbf{s} $\in \mathbb{H}(L)$ we have $H_*(A_{\scriptsize{\mbox{\textbf{s}}}}^-) \cong HF^-(S^3) = \mathbb{F}_{(0)}[U]$. For any \textbf{s}$'$ $\leq$ \textbf{s}, we can decrease the $s_j$ by one over finitely many steps to get from \textbf{s} to \textbf{s}$'$. By property $5$ we know that each of these steps will either preserve the top degree or drop it by $2$. The result now follows. 
\end{proof}
\begin{cor}\label{cor:inv}
For an $L$-space link $L \subset S^3$ with Heegaard diagram $\mathcal{H}$, $HC(\boldsymbol{A}^-(\mathcal{H},$ \textbf{s}$))$ depends only on $L$ and \textbf{s}.
\end{cor}
\begin{proof}
The top gradings of all the $H_*(A_{\scriptsize{\mbox{\textbf{s}}},\varepsilon}^-)$ are invariants of $L \subset S^3$ and \textbf{s}. The maps induced by homology in $\boldsymbol{A}^-(L,$ \textbf{s}$)^I$ are completely determined by these gradings since we have shown that $\boldsymbol{A}^-(\mathcal{H},$ \textbf{s}$)$ is a basic short exact cube. 
\end{proof}

Here is another fact that we will use often:
\begin{lem}\label{lem:hat}
Fix some $\mbox{\textnormal{\textbf{s}}} \in \mathbb{H}(L)$ where $L$ in $S^3$ is an arbitrary link (i.e. not necessarily an $L$-space link). Then, if $HFL^-(L,\mbox{\textnormal\textbf{s}}+\varepsilon)$ is trivial for every $\varepsilon \in \{0,1\}^l, \varepsilon \neq \boldsymbol{0}$ we get $HFL^-(L,\mbox{\textnormal{\textbf{s}}}) \cong \widehat{HFL}(L,\mbox{\textnormal{\textbf{s}}})$.
\end{lem}
\begin{proof}
First fix a Heegaard diagram $\mathcal{H}$ for $L \subset S^3$. We define an $l$-dimensional short exact cube $\boldsymbol{C}_{\mbox{\scriptsize{\textbf{s}}}}$ as follows: for $\varepsilon \in \{0,1\}^l$ and \textbf{s} $\in \mathbb{H}(L)$ we define $\boldsymbol{C}_{\mbox{\scriptsize{\textbf{s}}},\varepsilon}$ to be a quotient complex of $CF^-(\mathcal{H})$ generated by those $[\mbox{\textbf{x}},i_1,j_1,\ldots,i_l,j_l,$ $i_{l+1},\ldots,i_m]$ that satisfy the following:\\
\begin{itemize}
\item $i_k = 0$ and $j_k < s_k $ if $\varepsilon_k = 0$
\item $i_k = 0$ and $j_k \leq s_k$ if $\varepsilon_k = 1$
\item $i_k = 0$ and $j_k = s_k$ if $\varepsilon_k = 2$.
\end{itemize} 
Then the inclusion and quotient maps of $\boldsymbol{C}_{\mbox{\scriptsize{\textbf{s}}}}$ are defined naturally from $CF^-(\mathcal{H})$.

By definition, $H_*(\overline{\boldsymbol{C}}_{\mbox{\scriptsize{\textbf{s}}}}) \cong \widehat{HFL}(L,\mbox{\textnormal\textbf{s}})$ and for $\varepsilon \in \{0,1\}^l$ we have 
\[H_*(\boldsymbol{C}_{\mbox{\scriptsize{\textbf{s}}},\varepsilon}) \cong U_1^{1-\varepsilon_1}U_2^{1-\varepsilon_2}\ldots U_l^{1-\varepsilon_l}HFL^-(L,\mbox{\textbf{s}}+\boldsymbol{1}-\varepsilon).\]
And so $H_*(\boldsymbol{C}_\varepsilon)$ is only nonzero when $\varepsilon = \boldsymbol{1}$. So it follows by taking iterated quotients that,
\[HFL^-(L,\mbox{\textbf{s}}) \cong H_*(\boldsymbol{C}_{\boldsymbol{1}}) \cong H_*(\overline{\boldsymbol{C}}) \cong \widehat{HFL}(L,\mbox{\textnormal\textbf{s}}).\]
\end{proof}

\section{Proof of the main Theorems}\label{sec:Proof}
\begin{rem}\label{rem:shift} Suppose that $M = L_{i_1}\sqcup L_{i_2}\sqcup \ldots L_{i_k}$ is a sub-link of $L = L_1\sqcup L_2 \ldots \sqcup L_l$ with the inherited orientation. Fix some Heegaard diagram $\mathcal{H}$ for $L$. Now choose any \textnormal{\textbf{s}} $= (s_1, \ldots s_l) \in \mathbb{H}(L)$ so that all $s_j$ for $L_j \not\in M$ are sufficiently large (for instance larger than max$\{A_j($\textnormal{\textbf{x}}$)\}$ for every generator \textnormal{\textbf{x}} in some fixed diagram $\mathcal{H}$ for $L \subset S^3$). Then it is easy to see that for some \textbf{r}$\in \mathbb{H}(M)$ and any $\varepsilon \in E_l$ the complex $A_{\scriptsize{\mbox{\textnormal{\textbf{s}}}},\varepsilon}^-(\mathcal{H})$ is the same as $A_{\scriptsize{\mbox{\textnormal{\textbf{r}}}},\varepsilon'}^-(\mathcal{H}')$ where $\varepsilon'\in E_{l-k}$ is obtained from $\varepsilon$ by deleting $\varepsilon_{i_1}, \ldots, \varepsilon_{i_k}$ and reordering and $\mathcal{H}'$ is obtained by deleting $z_{i_1},\ldots, z_{i_k}$ and reordering. The explicit value for \textbf{r} can be computed by the formula in section $4.5$ of \cite{MOSurg} (see also section $3.7$ of \cite{OZLinks}). So \textnormal{\textbf{r}} $ = (r_1,\ldots, r_k) \in \mathbb{H}(M)$ is given by $r_j = s_{i_j} - \mbox{\textnormal{lk}}(L_{i_j}, L \char92 M)/2$. The next Lemma was observed in \cite{YLspace} Lemma $1.10$.
\end{rem}
\begin{lem}\label{lem:sublink} Every sub-link of an $L$-space link is an $L$-space link.
\end{lem}
\begin{proof}
Suppose that $M\subset L$ is some sub-link. It suffices to show that, for any \textbf{r} $\in \mathbb{H}(M)$, we have $H_*(A_{\scriptsize{\mbox{\textnormal{\textbf{r}}}}}^-(M)) \cong \mathbb{F}[U]$. This is true because $H_*(A_{\scriptsize{\mbox{\textnormal{\textbf{r}}}}}^-(M)) \cong H_*(A_{\scriptsize{\mbox{\textnormal{\textbf{s}}}}}^-(L))$ for some \textbf{s} $\in \mathbb{H}$ as shown above. 
\end{proof}
\begin{lem}\label{lem:PL}
\[\sum_{\scriptsize{\mbox{\textnormal{\textbf{s}}}}\in \mathbb{H}} \chi(HFL^-(L,\textnormal{\textbf{s}}))u^{\scriptsize{\mbox{\textnormal{\textbf{s}}}}} = P^L_\emptyset(u_1,\ldots,u_l)\]

\end{lem}
\begin{proof} It was shown in \cite{OZLinks} proposition 9.1 that when $l>1$
\[\sum_{\scriptsize{\mbox{\textnormal{\textbf{s}}}}\in \mathbb{H}} \chi(\widehat{HFL}(L,\textnormal{\textbf{s}}))u^{\scriptsize{\mbox{\textnormal{\textbf{s}}}}} = \pm \left(\prod_{i=1}^l u_i^{\frac{1}{2}} - u_i^{-\frac{1}{2}}\right)\Delta_L\]
and we have chosen sign conventions so that
\[\sum_{\scriptsize{\mbox{\textnormal{\textbf{s}}}}\in \mathbb{H}} \chi(\widehat{HFL}(L,\textnormal{\textbf{s}}))u^{\scriptsize{\mbox{\textnormal{\textbf{s}}}}} = \left(\prod_{i=1}^l u_i^{\frac{1}{2}} - u_i^{-\frac{1}{2}}\right)\Delta_L.\]
and so for $l>1$ it follows that
\begin{eqnarray*}
\sum_{\scriptsize{\mbox{\textnormal{\textbf{s}}}}\in \mathbb{H}} \chi(HFL^-(L,\textnormal{\textbf{s}}))u^{\scriptsize{\mbox{\textnormal{\textbf{s}}}}} & = & \left(\prod_{(a_1,\ldots, a_l)\in \mathbb{Z}_{\leq 0}^l} u_1^{a_1}\ldots u_l^{a_l}\right)\left(\sum_{\scriptsize{\mbox{\textnormal{\textbf{s}}}}\in \mathbb{H}} \chi(\widehat{HFL}(L,\textnormal{\textbf{s}}))u^{\scriptsize{\mbox{\textnormal{\textbf{s}}}}}\right) \\
& = &  \left(\prod_{(a_1,\ldots, a_l)\in \mathbb{Z}_{\leq 0}^l} u_1^{a_1}\ldots u_l^{a_l}\right)\left(\prod_{i=1}^l u_i^{\frac{1}{2}} - u_i^{-\frac{1}{2}}\right)\Delta_L\\
& = &  \left(\prod_{i=1}^l \frac{u_i^{\frac{1}{2}} - u_i^{-\frac{1}{2}}}{1-u_i^{-1}}\right) \Delta_L\\
& = & \sqrt{u_1u_2\ldots u_l} \Delta_L\\
& = & P^L_\emptyset.
\end{eqnarray*}
When $l=1$, it was shown in \cite{OZKnot} that:
\[\sum_{\scriptsize{\mbox{\textnormal{\textbf{s}}}}\in \mathbb{H}} \chi(\widehat{HFL}(L,\textnormal{\textbf{s}}))u^{\scriptsize{\mbox{\textnormal{\textbf{s}}}}} = \pm \Delta_L(u_1);\]
and so the result follows by the same argument as above. 
\end{proof}
\begin{defn}
Suppose we are given a Heegaard diagram $\mathcal{H}$ for an $L$-space link $L \subset S^3$. Define a directed labeled graph $\mathfrak{T}(\mathcal{H})$ as follows:
\begin{itemize}
\item The vertices correspond to the elements of $\mathbb{H}(L)$.
\item There is a directed edge from \textbf{s} $ = (s_1, \ldots, s_l)$ to \textbf{s}$' = (s'_1, \ldots, s'_l)$ if for some $i$ we have $s'_i = s_i+1$ and $s'_j = s_j$ for every $j \neq i$. We will call this edge $e_{\scriptsize{\mbox{\textbf{ss}}}'}$.
\item If \textbf{s} and \textbf{s}$'$, are as above then define $\varepsilon \in E_l$ so that $\varepsilon_j = 1$ if $j \neq i$ and $\varepsilon_i = 0$. Then the label of edge $e_{\scriptsize{\mbox{\textbf{ss}}}'}$ is the same as the label of the edge between $\varepsilon$ and \textbf{1} in $HC(\boldsymbol{A}^-(L,$ \textbf{s}$'))$.
\end{itemize}
Just as in corollary \ref{cor:inv}, the graph $\mathfrak{T}(\mathcal{H})$ is an invariant of $L\subset S^3$. So we will simply say $\mathfrak{T}(L)$. We will denote by $^j_s\mathfrak{T}(L)$ the subgraph of $\mathfrak{T}(L)$ that is obtained by restricting to the hyperplane with $j$th coordinate equal to $s$.
\end{defn}

\begin{defn} Suppose that $L\subset S^3$ is an $L$-space link. Then we recursively define $m(L) \in \mathbb{H}(L)$ as follows. If $L$ has only one component let $m(L)$ be the degree of $\Delta_L$. In general;
\[m(L)_i = \mbox{max}\left(\{\mbox{deg}_{u_i}(P^L_\emptyset)\}\cup \left\{\left. m(L\char92 L_j)_{i-1}+\frac{\mbox{lk}(L_i,L_j)}{2}\right\rvert j<i\right\} \cup \left\{\left. m(L\char92 L_j)_{i}+\frac{\mbox{lk}(L_i,L_j)}{2}\right\rvert j>i\right\}\right)\]
where by $\mbox{deg}_{u_i}(P^L_\emptyset)$ we mean the maximal degree of $u_i$ in any monomial of $P^L_\emptyset$. 
\end{defn}
\begin{prop}\label{prop:maxs}
For an $L$-space link $L \subset S^3$ suppose that $s \geq m(L)_j$. Then ${}^j_s\mathfrak{T}(L)$ is completely determined by $\mathfrak{T}(L\char92 L_j)$ and all the edges from ${}^j_s\mathfrak{T}(L)$ to ${}^{\;\;\;\;\;j}_{s+1}\mathfrak{T}(L)$ must be labeled with $0$.
\end{prop}

\begin{proof}
First note that $\mathfrak{T}(L\char92 L_j)$ only makes sense in light of Lemma \ref{lem:sublink} from which it follows that $L\char92 L_j$ is an $L$-space link. Pick \textbf{m} $= (m_1,\ldots,m_l) \in \mathbb{H}(L)$ so that for any $1\leq i \leq l$, $m_i > A_i(\mbox{\textbf{x}})$ for every generator $\mbox{\textbf{x}}$. Then, we claim that whenever $s_i > m_i$, ${}^{\;\;i}_{s_i}\mathfrak{T}(L)$ is completely determined by $\mathfrak{T}(L\char92 L_i)$ and all the edges from ${}^{\;\;i}_{s_i}\mathfrak{T}(L)$ to ${}^{\;\;\;\;\;\;i}_{s_i+1}\mathfrak{T}(L)$ must be labeled with $0$. We prove this claim when $i = l$. Since $s_l > m_l$, the inclusion between $A^{-}_{(s_1,\ldots,s_l)}$ and $A^{-}_{(s_1,\ldots,s_l+1)}$ induces an isomorphism on homology. So the edge between $(s_1,\ldots,s_l)$ and $(s_1,\ldots,s_l+1)$ is labeled with $0$. Following Remark \ref{rem:shift} we get that the edge between $(s_1, \ldots,s_i, \ldots, s_l)$ and $(s_1, \ldots,s_i+1, \ldots, s_l)$ has the same label as the edge between $\left(s_1 - \frac{\mbox{\scriptsize{lk}}(L_1, L_l)}{2}, \ldots, s_i - \frac{\mbox{\scriptsize{lk}}(L_i, L_l)}{2}, \ldots s_{l-1} - \frac{\mbox{\scriptsize{lk}}(L_{l-1}, L_l)}{2}\right)$ and $\left(s_1 - \frac{\mbox{\scriptsize{lk}}(L_1, L_l)}{2}, \ldots, s_i - \frac{\mbox{\scriptsize{lk}}(L_i, L_l)}{2}+1, \ldots, s_{l-1} - \frac{\mbox{\scriptsize{lk}}(L_{l-1}, L_l)}{2}\right)$ in $\mathfrak{T}(L\char92 L_l)$ and so this proves the claim.\\

Now we are ready to prove the proposition.\\

We will prove this by induction on $l$. If $m_j-1\geq m(L)_j$, for some fixed j, the edge between $(s_1,\ldots,m_j-1,\ldots,s_l)$ and $(s_1,\ldots,m_j,\ldots,s_l)$ is labeled zero if $s_i \geq m_i$ for every $i \neq j$ (by induction).\\

 Notice that this determines $\widetilde{HC}(\boldsymbol{A}^-(L,(s_1,\ldots,m_j,\ldots,s_l)))$. One valid (in the sense of Remark \ref{rem:path}) labeling of the remaining edges in $HC(\boldsymbol{A}^-(L,(s_1,\ldots,m_j,\ldots,s_l)))$ is given by setting all the edges between $HC(\boldsymbol{A}^-(L,(s_1,\ldots,m_j,\ldots,s_l))) \cap {}^{\;\;\;\;\;\;j}_{m_j-1}\mathfrak{T}(L)$ and $HC(\boldsymbol{A}^-(L,(s_1,\ldots,m_j,\ldots,s_l))) \cap {}^{\;\;j}_{m_j}\mathfrak{T}(L)$ to be zero and letting an edge between \textbf{s}$_1$ and \textbf{s}$_2$ in $HC(\boldsymbol{A}^-(L,(s_1,\ldots,m_j,\ldots,s_l))) \cap {}^{\;\;\;\;\;\;j}_{m_j-1}\mathfrak{T}(L)$ have the same labeling as the edge between \textbf{s}$_1'$ and \textbf{s}$_2'$ in $HC(\boldsymbol{A}^-(L,(s_1,\ldots,m_j,\ldots,s_l))) \cap {}^{\;\;j}_{m_j}\mathfrak{T}(L)$ where \textbf{s}$_1'$ and \textbf{s}$_2'$ are the same as \textbf{s}$_1$ and \textbf{s}$_2$ after adding one to the $j$th coordinate.

Since $m_j-1>\mbox{deg}_{u_j}P^L_\emptyset$ we must have $\chi(H_*(\overline{\boldsymbol{A}^-(L,(s_1,\ldots,m_j,\ldots,s_l))})) = 0$ and so the labeling for $HC(\boldsymbol{A}^-(L,(s_1,\ldots,m_j,\ldots,s_l)))$ described above is the correct one since it yields the correct Euler characteristic (see Remark \ref{lem:echar} and Lemma \ref{lem:HC}). We can similarly fill in all of ${}^{\;\;\;\;\;\;j}_{m_j-1}\mathfrak{T}(L)$ and all the edges between ${}^{\;\;\;\;\;\;j}_{m_j-1}\mathfrak{T}(L)$ and ${}^{\;\;j}_{m_j}\mathfrak{T}(L)$ are labeled $0$. Repeating this process by inductively decreasing the $j$th coordinate proves the claim. 
\end{proof}

\begin{lem} For a $2$ or $3$ component $L$-space link, $\mathfrak{T}(L)$ completely determines $HFL^-(L,$ \textnormal{\textbf{s}}$)$ for every \textnormal{\textbf{s}} $\in \mathbb{H}(L)$.
\end{lem}
\begin{proof}  Note that $\mathfrak{T}(L)$ determines all the hypercube graphs of $\boldsymbol{A}^-(L,$ \textbf{s}$)$ for any \textbf{s} $\in \mathbb{H}(L)$. Thus, by Lemma \ref{lem:comp} and Remark \ref{rem:path} we get that $\mathfrak{T}(L)$ determines all the $HFL^-(L,$ \textbf{s}$)$ upto an even shift in absolute grading. To fix the grading note that we can pick \textbf{s} $\in \mathbb{H}(L)$ so that any edge emerging from \textbf{s}$' \geq $ \textbf{s} is $0$ since for \textbf{s} sufficiently large $H_*(A_{\scriptsize{\mbox{\textbf{s}}}}^-(L)) \cong HF^-(S^3) = \mathbb{F}_{(0)}[U]$. This fixes the grading as required.
\end{proof}

\begin{lem} For an $L$-space link $L$, the graph $\mathfrak{T}(L)$ is determined by the polynomials $\pm\Delta_{M}$ and the linking numbers \textnormal{lk}$(L_i, M)$ where $M$ is any sublink of $L$. 
\end{lem}
\begin{proof} We will prove this by inducting on $l$. First suppose that $l = 1$. Then $\pm\Delta_L$ completely determines $\pm P^L_\emptyset = \sum_{s \in \mathbb{Z}} a_s(u_1)^s$. The only possibilities for $|a_s|$ are either $1$ or $0$.  If $|a_s| = 1$ then this forces the edge between $s-1$ and $s$ to be labeled with $1$. If $a_s = 0$ then this forces the edge between $s-1$ and $s$ to be labeled with $0$. This proves the case when $l = 1$.\\
By proposition \ref{prop:maxs}, we see that the subgraph of $\mathfrak{T}(L)$ that is induced by all the vertices \textbf{s} $= (s_1, \ldots, s_l)$ satisfying \textbf{s}$_i \geq m(L)_i$ for some $i$, is completely determined by the relevant polynomials and linking numbers. 

For the rest of $\mathfrak{T}(L)$ note that every edge of $\widetilde{HC}(\boldsymbol{A}^-(L,m(L)))$ is contained inside the part of the graph whose labels we have already determined. By Lemma \ref{lem:HC}, this either completely determines $HC(\boldsymbol{A}^-(L,m(L)))$, or all the edges emerging from $(m(L)_1-1,\ldots, m(L)_l-1)$ are labeled with a $0$ or they are all labeled with $1$. If $HC(\boldsymbol{A}^-(L,m(L)))$ is not completely determined by $\widetilde{HC}(\boldsymbol{A}^-(L,\mbox{\textbf{m}}))$, then we can use Lemma \ref{lem:echar} to see that the absolute values of the coefficients of $\Delta_L$ are enough to determine if all the edges emerging from $(m(L)_1-1,\ldots, m(L)_l-1)$ are labeled with a $0$ or $1$. Thus, we now have computed $\widetilde{HC}(\boldsymbol{A}^-(L,(m_1,\ldots,m_i-1,\ldots,m_l))$ for any $i$ and so we can proceed as before to inductively fill out all of $\mathfrak{T}(L)$. This proves the Lemma.
\end{proof}

\begin{proof}[Proof of Theorem \ref{thm:main}] This follows immediately from the previous two Lemmas
\end{proof}

\begin{lem} Let $S = \{i_1,\ldots, i_k\} \subsetneq \{1,\ldots, l\}$ and suppose that $\{j_1,\ldots,j_{l-k}\} = \{1,\ldots, l\} \char92 S$ where $j_a < j_b$ when $a<b$. Pick \textnormal{\textbf{s}} $\in \mathbb{H}(L)$ so the $s_{i_p} \geq m(L)_{i_p}$. Then if $a_{s_{j_1},s_{j_2},\ldots,s_{j_{l-k}}}$ is the coefficient of $u_{j_1}^{s_{j_1}}\ldots u_{j_{l-k}}^{s_{j_{l-k}}}$ in $P^L_{L_S}$, we have $a_{s_{j_1},s_{j_2},\ldots,s_{j_{l-k}}} = \chi(H_*(A_{\scriptsize{\mbox{\textnormal{\textbf{s}}}}, \varepsilon}^-(L)))$, where $\varepsilon \in E_l$ satisfies $\varepsilon_r = 2$ if $r = j_p$ for some $p$ and $\varepsilon_r = 1$ otherwise.
\end{lem}
\begin{proof} This follows from Remark \ref{rem:shift} and Lemma \ref{lem:PL}.
\end{proof}

\begin{proof}[Proof of Theorem \ref{thm:alex}] 
We will assume WLOG that $r = 1$. Then let $S = \{i_1,\ldots, i_k\} \subset \{2,\ldots, l\}$ and $\{j_1,\ldots,j_{l-k-1}\} = \{2,\ldots, l\} \char92 S$ with $j_a<j_b$ if $a<b$. \textbf{s} $ = \left(s_1,\ldots, s_l\right)\in \mathbb{H}\left(L\right)$ is arbitrary. Fix $\left(m_1,\ldots, m_l\right)\in \mathbb{H}\left(L\right)$ so that $m_i > m\left(L\right)_i + 1$. Then we have the following:
\[R_{\substack{\scriptsize{\mbox{\textbf{s}}'\geq \mbox{\textbf{s}}}\\s'_1 = s_1}}\left(P^L_{L_S}\right) 
=\sum_{\substack{\scriptsize{\mbox{\textbf{s}}'} = \left(s'_1, \ldots, s'_l\right) \in \mathbb{H}\\s'_1 = s_1,s'_{i_p} = m_{i_p}\\ m_{j_p} \geq s'_{j_p} \geq s_{j_p}}} \chi\left(H_*\left(A_{\scriptsize{\mbox{\textnormal{\textbf{s}}}'}, \rho}^-\right)\right),\]
where $\rho \in E_l$ is fixed and satisfies $\rho_k = 2$ if $k = j_p$ for some $p$, and $\rho_k = 1$ otherwise. This follows by the previous Lemma. We get that the above quantity is equal to:
\begin{equation}\label{eq:alex1}
 \sum_{\substack{\scriptsize{\mbox{\textbf{s}}'} = \left(s'_1, \ldots, s'_l\right) \in \mathbb{H}\\s'_1 = s_1,s'_{i_p} = m_{i_p}\\ m_{j_p} \geq s'_{j_p} \geq s_{j_p}}}\sum_{\substack{\varepsilon \in E_l,\varepsilon_1 = 2, \varepsilon_{i_p} = 1\\ \varepsilon_{j_p} = 1 \scriptsize{\mbox{ or }} 0}}\left(-1\right)^{\scriptsize{\mbox{number of $0$'s in $\varepsilon$}}}\chi\left(H_*\left(A_{\scriptsize{\mbox{\textnormal{\textbf{s}}}'}, \varepsilon}^-\right)\right).
\end{equation}
Note that if $\varepsilon \in E_l$ with $\varepsilon_1 = 2$, $\varepsilon_i = 0$ or $1$ if $i \neq 1$ we get:
\[A_{\scriptsize{\mbox{\textnormal{\textbf{s}}}'}, \varepsilon}^- = A_{\scriptsize{\mbox{\textnormal{\textbf{s}}}}'', \left(2,1,\ldots,1\right)}^-.\]
where \textbf{s}$''$ is given by \textbf{s}$''_1 =$ \textbf{s}$'_1$ and \textbf{s}$''_k = $\textbf{s}$'_k + \varepsilon_k -1$. 
So all of the terms in (\ref{eq:alex1}) that correspond to \textbf{s}$'$ with $s'_i \neq s_i$ or $m_i$ will cancel out. This leaves,
\begin{equation}\label{eq:alex2}
\sum_{\substack{\scriptsize{\mbox{\textbf{s}}'} = \left(s'_1, \ldots, s'_l\right) \in \mathbb{H}\\s'_1 = s_1,s'_{i_p} = m_{i_p}\\ s'_{j_p} =  s_{j_p} \scriptsize{\mbox{ or }} m_{j_p}}}\left(-1\right)^{\scriptsize{\mbox{number of $0$'s in $\nu\left(\mbox{\textbf{s}}'\right)$}}}\chi\left(H_*\left(A_{\scriptsize{\mbox{\textnormal{\textbf{s}}}'}, \nu\left(\mbox{\textbf{s}}'\right)}^-\left(L\right)\right)\right),
\end{equation}
where here $\nu\left(\mbox{\textbf{s}}'\right)_1 = 2, \nu\left(\mbox{\textbf{s}}'\right)_{i_p} = 1$ and $\nu\left(\mbox{\textbf{s}}'\right)_{j_p} = 1$ if $s'_{j_p} = m_{j_p}$, and $\nu\left(\mbox{\textbf{s}}'\right)_{j_p} = 0$ otherwise.\\
Given $S\subset \{2,\ldots, l\}$, we define \textbf{s}$\left(S\right)$ by setting \textbf{s}$\left(S\right)_1 = s_1,$ \textbf{s}$\left(S\right)_k = m_p$ if $p \in S$, and \textbf{s}$\left(S\right)_k = s_p - 1$ otherwise. Then we can rewrite (\ref{eq:alex2}) as
\begin{equation}\label{eq:alex3}
\sum_{S'\subset \{2,\ldots,l\}\char92 S} \left(-1\right)^{l-1-|S|-|S'|}\chi\left(H_*\left(A^-_{\mbox{\scriptsize{\textbf{s}}}\left(S\cup S'\right),\left(2,1,\ldots,1\right)}\right)\right).
\end{equation}
Thus, we finally get:
\begin{eqnarray}
\sum_{S\subset\{2,\ldots,l\}} \left(-1\right)^{l-1-|S|}R_{\substack{\scriptsize{\mbox{\textbf{s}}'\geq \mbox{\textbf{s}}}\\s'_1 = s_1}}\left(P^L_{L_S}\right)
&=&\sum_{S\subset\{2,\ldots,l\}}\sum_{S'\subset \{2,\ldots,l\}\char92 S} \left(-1\right)^{-|S'|}\chi\left(H_*\left(A^-_{\mbox{\scriptsize{\textbf{s}}}\left(S\cup S'\right),\left(2,1,\ldots,1\right)}\right)\right) \notag\\
&=&\sum_{S\subset\{2,\ldots,l\}}\sum_{A\subset S} \left(-1\right)^{-|S\char92 A|}\chi\left(H_*\left(A^-_{\mbox{\scriptsize{\textbf{s}}}\left(S\right),\left(2,1,\ldots,1\right)}\right)\right) \notag\\
&=&\chi\left(H_*\left(A^-_{\mbox{\scriptsize{\textbf{s}}}\left(\emptyset\right),\left(2,1,\ldots,1\right)}\right)\right). \label{eq:fin}
\end{eqnarray}
Now (\ref{eq:fin}) must be either $1$ or $0$ by Theorem \ref{thm:Aminus}.
\end{proof}

\section{Application to $2$-bridge links}\label{sec:Application}
We would like to use the recursive formula for the multivariate Alexander polynomial of a $2$-bridge link given in \cite{KBridge}, so we will use the conventions from that paper. A circle labeled $k$ or $-k$ will represent a braid with $k$ crossings as in Figure \ref{fig:twist}
\begin{figure}[h]
    \centering
    \includegraphics[scale = .6]{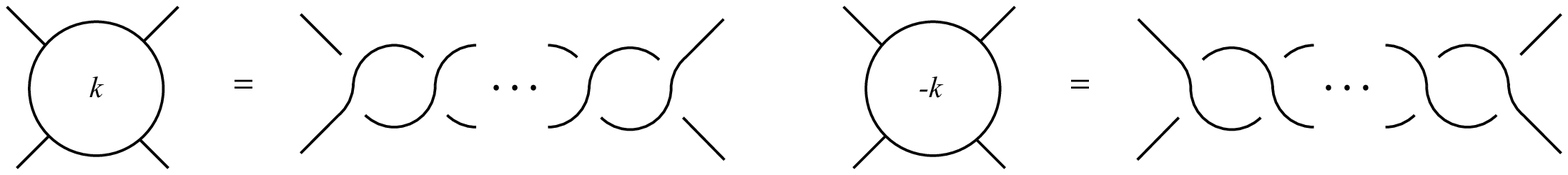}
		\caption{}
		\label{fig:twist}
\end{figure}
Suppose we are given a collection of nonzero integers $a_1,\ldots,a_n$. Then we can define $\alpha$ and $\beta$ via
\begin{equation}\label{eq:cont}
\frac{\alpha}{\beta} = a_1+\cfrac{1}{a_2+\cfrac{1}{\ddots + \cfrac{1}{a_n}}}
\end{equation}
where $\alpha > 0$, $\mbox{g.c.d}(\alpha,\beta) = 1$, and $\alpha > |\beta|>0$. Now, if $\alpha$ is even we can use $(a_1,\ldots,a_n)$ to construct an oriented link $C(a_1,\ldots,a_n)$ as shown in Figure \ref{fig:bridge}.

\begin{figure}[h]
    \centering
    \includegraphics[scale = .6]{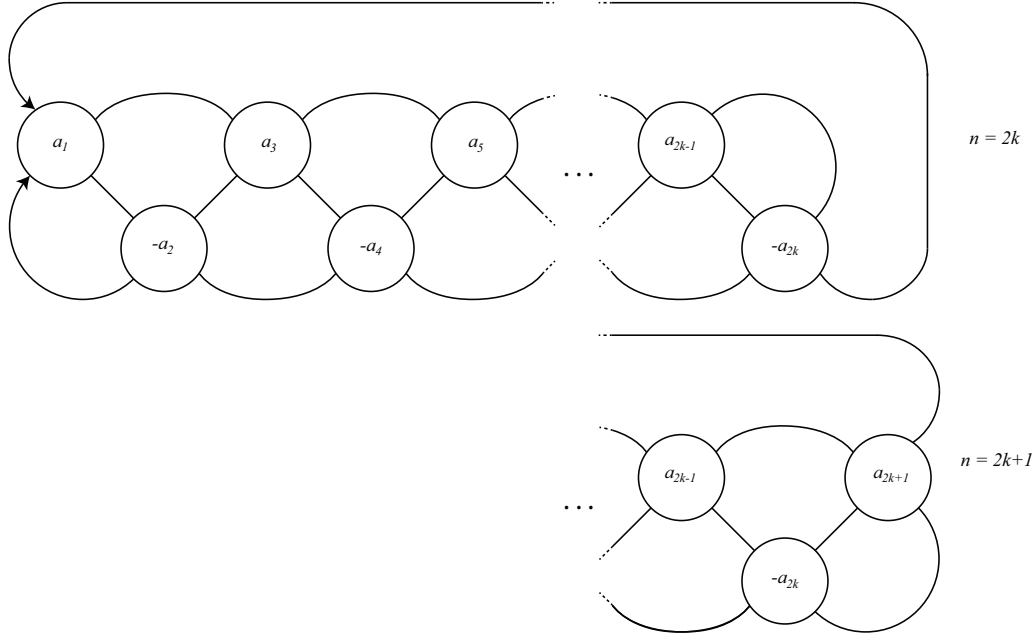}
		\caption{Diagram for constructing $2$-bridge link given a sequence of non-zero integers.}
		\label{fig:bridge}
\end{figure}

Links of this form are called $2$-bridge links, and we have the following classification from \cite{CEnum} and page 144 of \cite{SBridge} (see also chapter $12$ in \cite{GZKnots}):
\begin{thm}\label{thm:class}
If $L = C(a_1,\ldots, a_n)$ and $L' = C(b_1,\ldots, b_m)$ are two $2$ bridge links where we define $\alpha$ and $\beta$ from $a_1,\ldots,a_n$, as in equation \ref{eq:cont}, and similarly $\alpha'$ and $\beta'$ from $b_1,\ldots,b_m$. Then $L$ and $L'$ are equivalent iff $\alpha' = \alpha$ and $\beta' \equiv \beta^{\pm 1} \mod 2\alpha$. If $\beta' \equiv \beta + \alpha \mod 2\alpha$ or $\beta'\beta \equiv 1 + \alpha \mod 2\alpha$, then $L$ and $L'$ are equivalent after reversing the orientation of one of the components. 
\end{thm}

We will denote the $2$-bridge link determined by $\alpha$ and $\beta$ as above by $b(\alpha, \beta)$. To use the formulas in \cite{KBridge}, we need an expansion of $\frac{\alpha}{\beta}$ of the following form:
\[\frac{\alpha}{\beta} = 2p_1+\cfrac{1}{2q_1+\cfrac{1}{2p_2 + \cfrac{1}{2q_2  +\cfrac{1}{\ddots + \cfrac{1}{2p_n}}}}}.\]
We will denote $b(\alpha,\beta) = C(2p_1,2q_1,\ldots,2p_{n-1},2q_{n-1},2p_n)$ by $D(p_1,q_1,p_2,q_2,\ldots,p_n)$ for convenience.

We define two variable polynomials $F_r(u_1,u_2)$ for $r \in \mathbb{Z}$:
\[F_r(u_1,u_2)= 
\left\{
\begin{array}{ll}
\displaystyle\sum_{i = 0}^{r-1} (u_1u_2)^i &\mbox{ if } r>0\\
0 & \mbox{ if } r = 0\\
-\displaystyle\sum_{i=r}^{-1}(u_1u_2)^i & \mbox{ if } r < 0.
\end{array}
\right.\]

Now let us define polynomials $\Delta_k \in \mathbb{Z}[u_1^{\pm},u_2^{\pm}]$ for $0\leq k\leq n$ recursively as follows:
\[\Delta_0 = 0\]
\[\Delta_1 = F_{p_1}\]
\begin{equation}\label{eq:rec}
\Delta_k = (q_{k-1}(u_1-1)(u_2-1)F_{p_k}+1)\Delta_{k-1} + (u_1u_2)^{p_{k-1}}\frac{F_{p_k}}{F_{p_{k-1}}}(\Delta_{k-1} - \Delta_{k-2}).
\end{equation}
Also set $l_k = \sum_{i=1}^k p_i$ and $\tilde{l}_k = \sum_{i=1}^k |p_k|$. 
Then by Theorems $1,3$ and corollary $1$ of \cite{KBridge} we have;
\begin{thm}\label{thm:K} If $L = D(p_1,q_1,p_2,q_2,\ldots,p_k)$, then:
\[(u_1u_2)^{\frac{1-l_k}{2}}\Delta_k(u_1,u_2) = \pm\Delta_L(u_1,u_2).\]
The minimal degree of $u_1$ (or $u_2$) in any monomial of $\Delta_k$ is $\frac{l_k-\tilde{l}_k}{2}$ and the maximal degree of $u_1$ (or $u_2$) in any monomial of $\Delta_k$ is $\frac{l_k+\tilde{l}_k}{2}-1$.
\end{thm}
Define $q(k) = \prod_{i=1}^{k-1} q_i$ and $F(k) = \prod_{i=1}^k F_{p_i}$ where, as usual, the empty product is $1$. Also recall that the linking number of $D(p_1,q_1,p_2,q_2,\ldots,p_n)$ is $-l_n$.\\

Given any $P \in \mathbb{Z}[u_1^{\pm},u_2^{\pm}]$ where $P = \sum_{r,s \in \mathbb{Z}} a_{r,s}u_1^ru_2^s$, we define $P^{[i]}$ to be the polynomial $\sum_{j \in \mathbb{Z}} a_{j+i,j} u_1^{j+i}u_2^j$. If $P^{[i]} \neq 0$ we say that $P$ is supported on the diagonal $i$. Note that if $Q \in \mathbb{Z}[u_1^{\pm},u_2^{\pm}]$, then $(P+Q)^{[i]} = P^{[i]} + Q^{[i]}$ and $(PQ)^{[i]} = \sum_{a+b = i} P^{[a]}Q^{[b]}$ . Thus, it follows that if $P^{[0]}$ divides $Q$, then $(Q/P^{[0]})^{[k]} = Q^{[k]}/P^{[0]}$. Using equation (\ref{eq:rec}), we get the following identity:
\begin{equation}\label{eq:recd}
\Delta^{[k]}_n  = \sum_{i+j = k}(q_{n-1}(u_1-1)(u_2-1)F_{p_n}+1)^{[i]}\Delta_{n-1}^{[j]} + \left((u_1u_2)^{p_{n-1}}\frac{F_{p_n}}{F_{p_{n-1}}}\right)(\Delta_{n-1} - \Delta_{n-2})^{[k]}.
\end{equation}
This can then be expanded to:
\begin{align}\label{eq:recd2}
\Delta^{[k]}_n &= (q_{n-1}(-u_2)F_{p_n})\Delta_{n-1}^{[k+1]} + (q_{n-1}(-u_1)F_{p_n})\Delta_{n-1}^{[k-1]}+(q_{n-1}(u_1u_2+1)F_{p_n}+1)\Delta_{n-1}^{[k]}\notag\\
&\;\;+ \left((u_1u_2)^{p_{n-1}}\frac{F_{p_n}}{F_{p_{n-1}}}\right)(\Delta_{n-1} - \Delta_{n-2})^{[k]}.
\end{align} 
\begin{lem}
If $t>n-1$ then $\Delta_n$ is not supported on the diagonal $t$. Also:
\[\Delta_n^{[n-1]} = q(n)(-u_1)^{n-1}F(n).\] 
\end{lem}
\begin{proof}
First note that $\Delta_1^{[0]} = \Delta_1 = F_{p_1}$. Now the claim that $\Delta_n^{[t]} = 0$ when $t>n-1$ can be easily seen by induction via equation (\ref{eq:recd2}). We will prove that $\Delta_n^{[n-1]} = q(n)(-u_1)^{n-1}F(n)$ for $n>1$ by induction on $n$ using equation (\ref{eq:recd2}):
\begin{align*}
\Delta_n^{[n-1]} & = (q_{n-1}(-u_1)F_{p_n})\left(\prod_{i=1}^{n-2} q_i\right)(-u_1)^{n-2}\left(\prod_{i=1}^{n-1}F_{p_i}\right)\\
& = q(n)(-u_1)^{n-1}F(n).
\end{align*}
\end{proof}
\begin{lem} For $n\geq 2$:
\[\Delta_n^{[n-2]} = P_1 + P_2 + P_3\]
where:
\begin{align}\label{eq:n-2}
P_1 &= (n-1)(u_1u_2 +1)q(n)F(n)(-u_1)^{n-2}\notag\\
P_2 &= \sum_{i=2}^n \frac{q(n)}{q_{i-1}}\frac{F(n)}{F_{p_i}}(-u_1)^{n-2} \mbox{ and}\notag\\
P_3 &= \sum_{i=1}^{n-1}(u_1u_2)^{p_i}\frac{q(n)}{q_i}\frac{F(n)}{F_{p_i}}(-u_1)^{n-2}.
\end{align}
\end{lem}
\begin{proof} When $n=2$, we directly compute that:
\[\Delta_2 = q_1(u_1-1)(u_2-1)F_{p_2}F_{p_1} + F_{p_1} + (u_1u_2)^{p_1}F_{p_2}.\]
For $n> 2$, we can recursively compute $\Delta_n^{[n-2]}$: 
\begin{align*}
\Delta_n^{[n-2]} &= (q_{n-1}(u_1-1)(u_2-1)F_{p_n}+1)^{[0]}\Delta_{n-1}^{[n-2]} + (q_{n-1}(u_1-1)(u_2-1)F_{p_n}+1)^{[1]}\Delta_{n-1}^{[n-3]}\\
&\;\; +  (u_1u_2)^{p_{n-1}}\frac{F_{p_n}}{F_{p_{n-1}}}(\Delta_{n-1}^{[n-2]})\\
&= (q_{n-1}(u_1u_2+1)F_{p_n}+1)\frac{q(n)}{q_{n-1}}\frac{F(n)}{F_{p_n}}(-u_1)^{n-2} + q_{n-1}(-u_1)F_{p_n}\Delta_{n-1}^{[n-3]}\\
&\;\; + (u_1u_2)^{p_{n-1}}\frac{q(n)}{q_{n-1}}\frac{F(n)}{F_{p_{n-1}}}(-u_1)^{n-2}.
\end{align*}
The result now follows by induction.
\end{proof}
\begin{lem}\label{lem:polyres}
Let $\Delta_n = \sum_{i,j} a_{ij}u_1^iu_2^j$. Suppose that all the nonzero $a_{ij}$ are $\pm 1$. Suppose also that for fixed $i'$ (or $j'$) the nonzero $a_{i'j}$ (or $a_{j'i}$) alternate in sign. Then we must have $|q_i| = 1$ for every $1\leq i\leq n-1$. For the $p_i$, one of the following two possibilities holds:
\begin{itemize}
\item For $i \neq 1$ all $p_i$ are equal.  For $i \neq 1$, $p_i = \pm 1$ and $p_i = -q_{i-1}$
\item For $i \neq n$ all $p_i$ are equal. For $i \neq n$, $p_i = \pm 1$ and $p_i = -q_{i}$.
\end{itemize}
\end{lem}
\begin{proof}
First note that when $n=1$, the Lemma is vacuously true. So from now on we will assume that $n\geq 2$.
If $\Delta_n$ has all coefficients $\pm 1$ or $0$, then so does $\Delta_n^{[n-1]} = q(n)F(n)(-u_1)^{n-1}$. For this to happen $|q(n)|$ must be $1$ which implies that $q_i = \pm 1$ for every $1\leq i \leq n-1$. $F(n)$ has coefficients $\pm 1$ if for all but possibly one $i$, we have $p_i = \pm 1$.\\
Now we focus on $\Delta_n^{[n-2]}$. There are four cases:
\begin{case}[There is some $k \in \{1,2,\ldots, n\}$ such that $p_k > 1$]

Suppose that $r$ of the $p_i$ are $-1$ (and so except for $p_k$, the rest are $1$.) First, we get that:
\[F(n) = (-1)^r\sum_{i=0}^{p_k-1} (u_1u_2)^{i-r}.\]
Now, since all the nonzero coefficients of $\Delta_n$ are by assumption $\pm 1$, the same must be true for $\frac{\Delta_n^{[n-2]}}{q(n)(-u_1u_2)^{-r}(-u_1)^{n-2}}$. We will compute the coefficient of $u_1u_2$ in $\frac{\Delta_n^{[n-2]}}{q(n)(-u_1u_2)^{-r}(-u_1)^{n-2}}$. Now recall that
\[\Delta_n^{[n-2]} = P_1+P_2+P_3\]
where $P_1$, $P_2$ and $P_3$ are as defined in equation (\ref{eq:n-2}). Set $P'_i := \frac{P_i}{q(n)(-u_1u_2)^{-r}(-u_1)^{n-2}}$ for $i = 1,2$ or $3$. Then,
\[P'_1= (n-1)(u_1u_2 +1)\sum_{i=0}^{p_k-1}(u_1u_2)^i,\]
and so the coefficient of $(u_1u_2)$ in $P'_1$ is $2(n-1)$. Similarly,
\[P'_2 = q_{k-1} + \sum_{\substack{p_j =1\\2\leq j \leq n}} (q_{j-1})\sum_{i=0}^{p_k-1} (u_1u_2)^i + \sum_{\substack{p_j =-1\\2\leq j \leq n}} (-q_{j-1})\sum_{i=1}^{p_k} (u_1u_2)^i.\]
So the coefficient of $(u_1u_2)$ in $P'_2$ is
\[\sum_{\substack{2\leq j \leq n\\ j \neq k}} p_jq_{j-1},\]
and similarly the coefficient of $(u_1u_2)$ in $P'_3$ is 
\[\sum_{\substack{1\leq j \leq n-1\\ j \neq k}} p_jq_{j}.\]
So finally, the coefficient of $u_1u_2$ in $\frac{\Delta_n^{[n-2]}}{q(n)(-u_1u_2)^{-r}(-u_1)^{n-2}}$ is
\begin{equation}\label{eq:rand}
2(n-1) + \sum_{\substack{2\leq j \leq n\\ j \neq k}} p_jq_{j-1} + \sum_{\substack{1\leq j \leq n-1\\ j \neq k}} p_jq_{j},
\end{equation}
which must be $1$, $-1$ or $0$. Notice first that, if $2\leq k \leq n-1$ then the sum 
\[\sum_{\substack{2\leq j \leq n\\ j \neq k}} p_jq_{j-1} + \sum_{\substack{1\leq j \leq n-1\\ j \neq k}} p_jq_{j}\]
is bounded above in absolute value by $2(n-2)$, which makes it impossible for equation (\ref{eq:rand}) to be equal to $1$, $-1$ or $0$. So, we get that $k$ must be $1$ or $n$. If $k$ is $1$ then equation (\ref{eq:rand}) becomes
\[2(n-1) + p_nq_{n-1} + \sum_{j=2}^{n-1} p_j(q_j + q_{j-1}).\]
Notice that the above quantity has smallest possible value $1$ and this only occurs if all of the $q_i$ are equal and have opposite sign as all the $p_{i+1}$, which proves the claim in this case. When $k = n$ the argument is similar.
\end{case}
\begin{case}[There is some $k \in \{1,2,\ldots,n\}$ such that $p_k < -1$]
The argument is the same as in the previous case, except we divide $\Delta_n^{[n-2]}$ by $q(n)(-u_1u_2)^{-r}(-u_1)^{n-2}$ and examine the coefficient of $(u_1u_2)^{-1}$.
\end{case}
\begin{case}[All of the $p_i$ are $\pm 1$ and $n\geq 3$] We will start by showing that all the $q_i$ are equal. Suppose as in the previous cases that the number of $p_i$ that are $-1$ is $r$. In this case $\Delta_n^{[n-1]}$ is the monomial
\[(-1)^{n-1+r}q(n)u_1^{n-1-r}u_2^{-r} \neq 0.\]
This has the maximal possible degree for $u_1$ and minimal possible degree for $u_2$ by Theorem \ref{thm:K}. This immediately forces $\Delta_n^{[n-2]}$ to have at most $2$ nonzero coefficients, and $\Delta_n^{[n-3]}$ to have at most $3$ nonzero coefficients. So $\Delta_n^{[n-2]}$ is of the form 
\[a_{n-2-r,-r}u_1^{n-2-r}u_2^{-r} + a_{n-1-r,1-r}u_1^{n-1-r}u_2^{1-r}\]
Using the symmetry of the Alexander polynomial under the involution $u_i\mapsto u_i^{-1}$, as well as the symmetry given by exchanging $u_1$ and $u_2$ (there is an isotopy of $S^3$ exchanging the two components of a $2$ bridge link which is easy to see using the Schubert normal form \cite{SBridge}); we can conclude that $a_{n-2-r,-r} = a_{n-1-r,1-r}$. Suppose that $a_{n-2-r,-r} = a_{n-1-r,1-r} \neq 0$. Then since we have required the signs of $a_{i,j}$ to be alternating for fixed $i$ (and $j$), this forces one of the following possibilities for $\Delta_n^{[n-3]}$
\[\Delta_n^{[n-3]} = \pm(u_1^{n-3-r}u_2^{-r} + u_1^{n-2-r}u_2^{1-r} + u_1^{n-1-r}u_2^{2-r}) \mbox{ or } \pm(u_1^{n-2-r}u_2^{1-r}) \mbox{ or } 0.\]
We have ruled out $\pm(u_1^{n-3-r}u_2^{-r} + u_1^{n-1-r}u_2^{2-r})$ due to Theorem $3$ (see also definition $2$(iv)) in \cite{KBridge}. In all the possibilities for $\Delta_n^{[n-3]}$, we have 
\[\Delta_n^{[n-3]}(-1,1) = \pm 1 \mbox{ or } 0.\]
$F_{p_n}(-1,1)$ is always $1$ since we have assumed $p_n = \pm 1$. From this we conclude
\[\Delta_n^{[n-1]}(-1,1) = q(n) \mbox{ and } \Delta_n^{[n-2]}(-1,1) = 0.\]
Using this in the recursive formula for $\Delta_n^{[n-3]}$ given in equation (\ref{eq:recd2}), we get
\[\Delta_n^{[n-3]}(-1,1) = -q(n) + q(n-2) + q_{n-1}\Delta^{[n-4]}_{n-1}(-1,1).\]
We manually compute $\Delta_3^{[0]} = 1 - 2q_1q_2$. So this gives the formula
\[\Delta_n^{[n-3]}(-1,1) = \sum_{i=1}^{n-2} \frac{q(n)}{q_iq_{i+1}} -(n-1)q(n).\]
If the above sum is to equal $\pm 1$ (note that it cannot be $0$), we must have 
\[\sum_{i=1}^{n-2} \frac{1}{q_iq_{i+1}} = n-2,\]
and this can only happen if all the $q_i$ are equal.\\
Now suppose that $a_{n-2-r,-r} = a_{n-1-r,1-r} = 0$. The constant term of $\frac{\Delta_n^{[n-2]}}{q(n)(-u_1u_2)^{-r}(-u_1)^{n-2}}$ is:
\begin{equation}\label{eq:piqi}
(n-1) + \sum_{\substack{2\leq i\leq n\\p_i = 1}} q_{i-1} + \sum_{\substack{1\leq i\leq n-1\\p_i = -1}} (-q_i),
\end{equation}
which by our assumption must be $0$. We can rewrite (\ref{eq:piqi}) as;
 \begin{equation}\label{eq:piqi2}
(n-1) + \frac{q_{n-1}p_n+q_{n-1}}{2} + \frac{q_1p_1-q_1}{2}+ \sum_{\substack{2\leq i\leq n-1}} \frac{q_{i-1}p_i+q_ip_i+q_{i-1}-q_i}{2},
\end{equation}
which simplifies to
 \begin{equation}\label{eq:piqi3}
(n-1) + \frac{q_{n-1}p_n+q_1p_1}{2}+ \sum_{\substack{2\leq i\leq n-1}} \frac{q_{i-1}p_i+q_ip_i}{2}.
\end{equation}
Note that
\begin{equation}\label{eq:piqi4}
\frac{q_{n-1}p_n+q_1p_1}{2}+ \sum_{\substack{2\leq i\leq n-1}} \frac{q_{i-1}p_i+q_ip_i}{2}
\end{equation}
has a maximum absolute value of $n-1$ which can only happen if all the $q_i$ are equal (and have opposite sign as all the $p_i$).\\
So we have shown in all cases that all the $q_i$ are equal. This allows us to rewrite equation \ref{eq:piqi3} (which is the constant term of $\frac{\Delta_n^{[n-2]}}{q(n)(-u_1u_2)^{-r}(-u_1)^{n-2}}$) as; 
\begin{equation}\label{eq:piqi5}
(n-1) + \sum_{i=2}^{n-1} q_1p_i + q_1\left(\frac{p_1+p_n}{2}\right).
\end{equation}
We must have (\ref{eq:piqi5}) equal to $\pm 1$ or $0$. First note that we cannot have $q_1\frac{p_1+p_n}{2} = 1$ since $\sum_{i=2}^{n-1} q_1p_i$ is bounded above in absolute value by $n-2$. So we must have that $q_1\frac{p_1+p_n}{2}= -1$ or $0$. If $q_1\frac{p_1+p_n}{2}= 0$ then $\sum_{i=2}^{n-1} q_1p_i$ must be $-n+2$ which implies that all the $p_i$ for $2\leq i\leq n-1$ have the opposite sign as $q_1$ and since $q_1\frac{p_1+p_n}{2}= 0$ we get that one of $p_1$ and $p_n$ must also have the opposite sign as $q_1$ which proves the claim in this case. If we assume that $q_1\frac{p_1+p_n}{2}= -1$ then we need $\sum_{i=2}^{n-1} q_1p_i \leq 3-n$. However $\sum_{i=2}^{n-1} q_1p_i = 3-n$ is impossible since changing the $p_i$ always changes the sum $\sum_{i=2}^{n-1} q_1p_i$ by a multiple of $2$. Thus we once again have that $\sum_{i=2}^{n-1} q_1p_i = 2-n$. This along with the fact that  $q_1\frac{p_1+p_n}{2}= -1$ implies that all of the $p_i$ have the opposite sign as $q_1$.
\end{case}
\begin{case}[$n=2$ and all the $p_i$ are $\pm 1$]
The only tuples $(p_1,q_1,p_2)$ that do not satisfy the condition given in the Lemma are $(1,1,1)$ and $(-1,-1,-1)$, and we can manually compute $\Delta_2$ in both these cases to check that they do not satisfy that all of the nonzero coefficients are $\pm 1$. In particular for $(1,1,1)$ we have $\Delta_2 = 2-u_1-u_2 + 2u_1u_2$ and for $(-1,-1,-1)$ we have $\Delta_2 = -\frac{2}{u_1^2 u_2^2}+\frac{1}{u_1^2 u_2}+\frac{1}{u_1 u_2^2}-\frac{2}{u_1 u_2}$
\end{case}
\end{proof}

Now, if an oriented $2$-bridge link $L$ is an $L$-space link, it must satisfy the conditions of the Lemma \ref{lem:polyres} by corollary \ref{cor:alex2} and so if $L = D(p_1,q_1,\ldots,p_{n-1},q_{n-1},p_n)$, then we have narrowed things down to the following $8$ possibilities where $w >0$ is an integer, $q:= 2w+1, q':= 2w-1$ and $k:= 2n-1$.
\begin{align*}
L &= D(-1,1,\ldots,-1,1,w) = b(qk-1,q-(qk-1)) &\mbox{ or}\\
&=D(-1,1,\ldots,-1,1,-w) = b(q'k+1,q'-(q'k+1)) &\mbox{ or}\\
&=D(1,-1,\ldots,1,-1,w) = b(q'k+1,q'k+1-q') &\mbox{ or}\\
&=D(1,-1,\ldots,1,-1,-w) = b(qk-1,qk-1-q) &\mbox{ or}\\
&=D(w,-1,1,\ldots,-1,1) = b(q'k+1,k) &\mbox{ or}\\
&=D(-w,-1,1,\ldots,-1,1) = b(qk-1,-k) &\mbox{ or}\\
&=D(w,1,-1,\ldots,1,-1) = b(qk-1,k) &\mbox{ or}\\
&=D(-w,1,-1,\ldots,1,-1) = b(q'k+1,-k). &
\end{align*}
We can further reduce these $8$ possibilities down to $4$ by noting $b(qk-1,\pm k) = b(qk-1, \pm(q-(qk-1)))$ which can be seen by rotating the diagram given by \ref{fig:bridge} by $180^{\circ}$, and similarly $b(q'k+1,\pm k) = b(q'k+1, \pm(q'k+1-q'))$. Now we compute the signatures of these four possibilities.
\begin{lem}\label{lem:sig}
When $q$,$q'$ and $k$ are odd positive integers and $q \neq 1$ if $k = 1$;
\begin{align}
\sigma(b(qk-1,\pm k)) &= \pm(q-2)\\
\sigma(b(q'k+1,\pm k)) &= \pm q'.
\end{align}
\end{lem}
\begin{proof}
First we compute the signature of $b(q'k+1,k)$. Since $\frac{q'k+1}{k} = q' +\frac{1}{k}$, we can use Figure \ref{fig:bridge} to give a diagram $D$ for $b(qk-1,k)$. Now we will use the Gordon-Litherland formula for knot signature(see \cite{GLSig}) on $D$. Since the surface given by a checkerboard coloring of $D$ is orientable, the signature of the link is simply the signature of the Goeritz matrix for $D$ (see the end of the first page in \cite{GLSig}). We denote by $A_n(p)$ the $n \times n$ matrix with $A_{11} = p$, $A_{ii} = 2$ when $2\leq i \leq n$, $A_{ij} = -1$ when $|j-i| = 1$ and $0$ everywhere else. A Goeritz matrix for $D$ is given by $A_q(1+k)$. We claim that if $p>1$, $A_n(p)$ has signature $n$. This is easy to see inductively; let $B(p) = \begin{pmatrix}1&0\\ \frac{1}{p}&1\end{pmatrix}$, $I_n$ denote the $n\times n$ identity matrix and $B_n(p) = \begin{pmatrix}B(p)&0\\0&I_{n-2}\end{pmatrix}$. Then 
\[B_n(p)A_n(p)B_n(p)^T = \begin{pmatrix}p&0\\0&A_{n-1}(2-1/p)\end{pmatrix}\]
so $\sigma(A_n(p)) = 1+\sigma(A_n(2-1/p))$ and the claim follows. So the signature of $b(q'k+1,k)$ is $q'$. Since $b(q'k+1,-k)$ is the mirror image of $b(q'k+1,k)$, the signature of $b(q'k+1,-k)$ is $-q'$.\\
Now we consider $b(qk-1,k)$ where $k > 1$ ($k=1$ has already been covered above). $\frac{qk-1}{k} = q-\frac{1}{k}$. In this case a Goeritz matrix is $A_q(1-k)$ and 
\[B_q(1-k)A_q(1-k)B_q(1-k)^T = \begin{pmatrix}1-k&0\\0&A_{q-1}(2-1/(1-k))\end{pmatrix}.\]
Now $1-k<0$ and $2-1/(1-k)>1$, so $\sigma(A_n(1-k)) = -1+ \sigma(A_{q-1}(2-1/(1-k))) = q-2$. Since $b(qk-1,-k)$ is the mirror image of $b(qk-1,k)$, $\sigma(b(qk-1,-k)) = -q+2$ as desired. 
\end{proof}

\begin{prop}
If $L$ is an $L$-space link of the form $b(qk-1,k) = D(-1,1,\ldots,-1,1,w)$ then $L = b(2,1)$ 
\end{prop}
\begin{proof}
Let us assume that $L = b(qk-1,k)$ is an $L$-space link. Now if $s < n$, it is easy to see by induction that
\[\Delta_s(u_1,u_2) = -\frac{1}{u_1^su_2^s}\left(\sum_{i=0}^{s-1} u_1^iu_2^{s-1-i}\right).\]
So by equation (\ref{eq:rec}) we get 
\begin{multline*}
\Delta_n(u_1,u_2) = \left((u_1-1)(u_2-1)\left(\sum_{i=0}^{w-1}(u_1u_2)^i\right)+1\right)\left(-\frac{1}{u_1^{n-1}u_2^{n-1}}\left(\sum_{i=0}^{n-2} u_1^iu_2^{n-2-i}\right)\right) +\\
\frac{(u_1u_2)^{-1}}{-(u_1u_2)^{-1}}\left(\sum_{i=0}^{w-1}(u_1u_2)^i\right)\left(\left(-\frac{1}{u_1^{n-1}u_2^{n-1}}\left(\sum_{i=0}^{n-2} u_1^iu_2^{n-2-i}\right)\right)-\left(-\frac{1}{u_1^{n-2}u_2^{n-2}}\left(\sum_{i=0}^{n-3} u_1^iu_2^{n-3-i}\right)\right)\right).
\end{multline*}
This simplifies to
\[\Delta_n(u_1,u_2) = \sum_{\substack{0\leq i \leq w-1\\0\leq j\leq n-1}} u_1^{i+j+1-n}u_2^{i-j} - \sum_{\substack{0 \leq i \leq w\\0 \leq j \leq n-2}}u_1^{i+j+1-n}u_2^{i-j-1}.\]
Now note that $L = L_1 \sqcup L_2$, where both $L_1$ and $L_2$ are unknots and lk$(L_1,L_2) = -l_n = -w+n-1$, so we get:
\[P^L_{L_1}(u_2) = (u_2)^{\frac{n-w-1}{2}} \sum_{i=0}^\infty (u_2)^{-i} \mbox{ and } P^L_{L_2}(u_1) = (u_1)^{\frac{n-w-1}{2}} \sum_{i=0}^\infty (u_1)^{-i}.\]
Finally, by Theorem \ref{thm:K} we also get
\[P^L_\emptyset = \pm(u_1u_2)^{\frac{n-w+1}{2}}\Delta_n(u_1,u_2).\]
Expanding this then gives
\[\pm P^L_\emptyset = (u_1u_2)^{\frac{n-w+1}{2}}\Delta_n(u_1,u_2) = \sum_{\substack{0\leq i \leq w-1\\0\leq j\leq n-1}} u_1^{i+j+\frac{3-n-w}{2}}u_2^{i-j+\frac{n-w+1}{2}} - \sum_{\substack{0 \leq i \leq w\\0 \leq j \leq n-2}}u_1^{i+j+\frac{3-n-w}{2}}u_2^{i-j+\frac{n-w-1}{2}}.\]
If $n = 1$, we get:
\[\pm P^L_\emptyset = \sum_{0\leq i \leq w-1} u_1^{i+1\frac{-w}{2}}u_2^{i+1\frac{-w}{2}}.\]
We can then fix the sign for $P^L_\emptyset$ using corollary \ref{cor:alex2} to get
\[P^L_\emptyset = -\sum_{0\leq i \leq w-1} u_1^{i+1\frac{-w}{2}}u_2^{i+1\frac{-w}{2}}.\]
Then, using the method given in the proof of Theorem \ref{thm:main}, we can compute $\mathfrak{T}(L)$. In this case $m(L) = (w/2,w/2)$. The edge between $(s_1,w/2-1)$ and $(s_1,w/2)$ is labeled with $0$ whenever $s_1\geq w/2$. Similarly, the edge between $(w/2-1,s_2)$ and $(w/2,s_2)$ is labeled $0$ whenever $s_2\geq w/2$. The coefficient of $u_1^{w/2}u_2^{w/2}$ in $P^L_\emptyset$ is $-1$, which forces both edges from $(w/2-1,w/2-1)$ to be labeled with $1$. This along with Lemma \ref{lem:hat} allows us to compute
\begin{equation}\label{eq:a1}
\widehat{HFL}\left(L,\left(\frac{w}{2},\frac{w}{2}\right)\right)\cong \mathbb{F}_{(1)}.
\end{equation}
Now, recall that when $L$ is alternating, $\widehat{HFL}(L,\mbox{\textbf{s}})$ is completely determined by its Euler characteristic and $\sigma(L)$, using Theorem $1.3$ in \cite{OZLinks}. Specifically, if \textbf{s} $=(s_1,s_2)$ and $a_{\mbox{\scriptsize{\textbf{s}}}}$ is the coefficient of $u^{\mbox{\scriptsize{\textbf{s}}}}$ in $(1-u_1^{-1})(1-u_2^{-1})P^L_\emptyset$ then 
\[\widehat{HFL}(L,\textbf{s}) \cong \mathbb{F}^{|a_{\mbox{\scriptsize{\textbf{s}}}}|}_{s_1+s_2+\frac{\sigma -1}{2}}.\]
Therefore
\begin{equation}\label{eq:a2}
\widehat{HFL}\left(L,\left(\frac{w}{2},\frac{w}{2}\right)\right) \cong \mathbb{F}_{(2w-1)}
\end{equation}
by Lemma \ref{lem:sig}. Combining equations (\ref{eq:a1}) and (\ref{eq:a2}) gives $w = 1$, which along with $n = 1$, gives that $L = b(2,1)$.\\

If $n \neq 1$, the leading coefficient of $P^L_{\emptyset}|_{(1,j)}$ and $P^L_{\emptyset}|_{(1,j+1)}$ have opposite sign iff $j = \frac{w-n+1}{2}$, or in other words there is a sign change in the leading coefficients of $P^L_{\emptyset}|_{(1,j)}$ at $j = \frac{w-n+1}{2}$. Also note that in $P^L_{L_2}|_{(1, j)} = 0$ if $j>\frac{n-w-1}{2}$ and $u_1^j$ otherwise. Combining these facts using corollary \ref{cor:alex2}, we must have $w = n-1$. When $w = n-1$ we fix the sign of $P^L_{\emptyset}$ using corollary \ref{cor:alex2} to get
\[P^L_\emptyset =  \sum_{\substack{0\leq i \leq n-2\\0\leq j\leq n-1}} u_1^{i+j+\frac{3-n-w}{2}}u_2^{i-j+\frac{n-w+1}{2}} - \sum_{\substack{0 \leq i \leq n-1\\0 \leq j \leq n-2}}u_1^{i+j+\frac{3-n-w}{2}}u_2^{i-j+\frac{n-w-1}{2}}.\]
We now know enough to compute $\mathfrak{T}(L)$. We will compute the part of $\mathfrak{T}(L)$ inside the region bounded by $s_1+s_2 \geq n-2, s_1\geq 0$ and $s_2 \geq 0$. This is shown in Figure \ref{fig:case1}. 
\begin{figure}[t]
    \centering
    \includegraphics{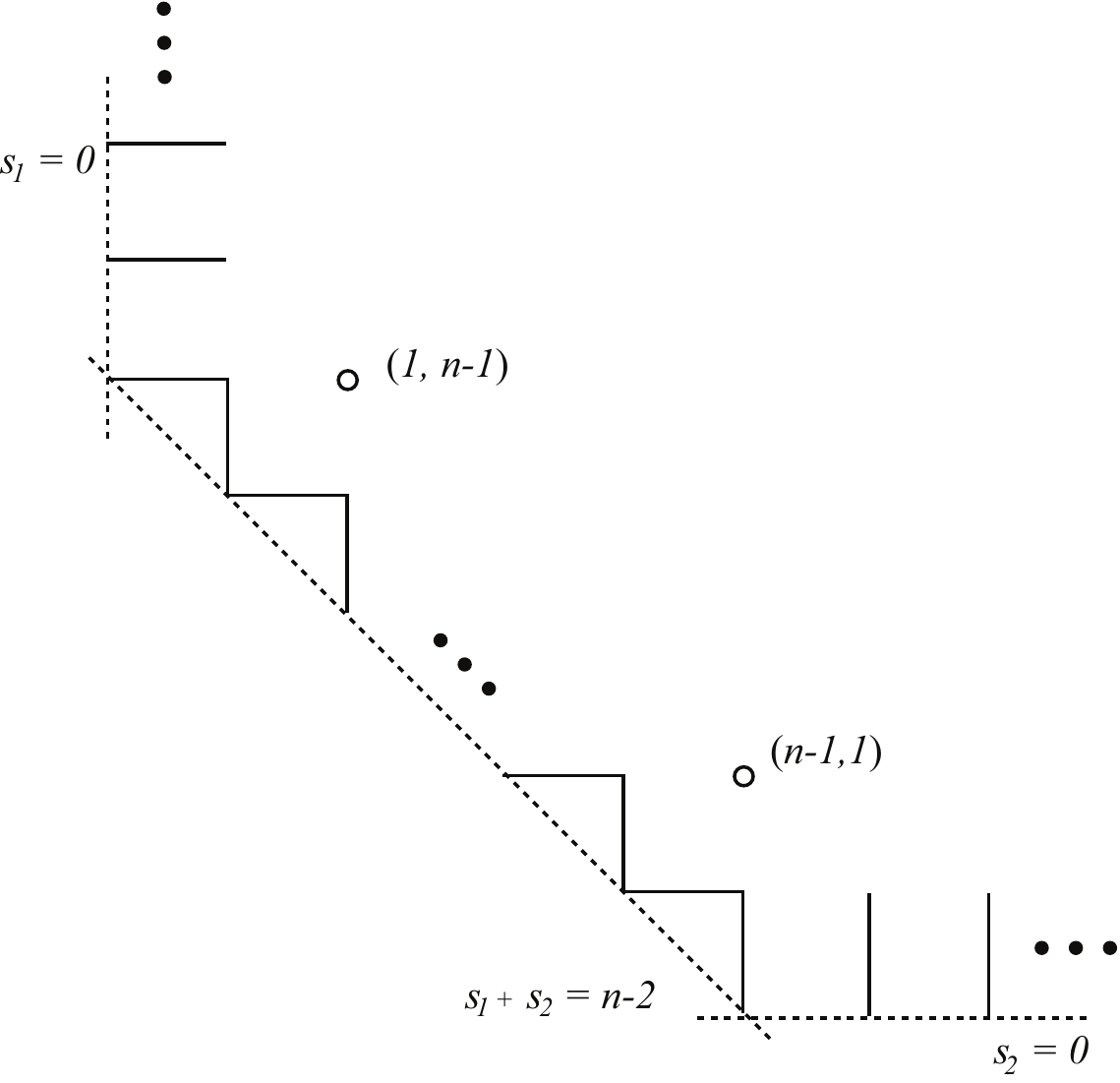}
    \caption{Part of $\mathfrak{T}(L)$, for $b(k^2-1,k)$ assuming it is an $L$-space link. Edges labeled with $1$ are drawn in black and edges labeled with $0$ are not shown.}
		\label{fig:case1}
\end{figure}
Using this and Lemma \ref{lem:hat} we compute
\begin{equation}\label{eq:a3}
\widehat{HFL}(L,(1,n-1)) \cong \mathbb{F}_{(1)}.
\end{equation}
Once again, using Theorem 1.3 in \cite{OZLinks}: if \textbf{s} $=(s_1,s_2)$ and $a_{\mbox{\scriptsize{\textbf{s}}}}$ is the coefficient of $u^{\mbox{\scriptsize{\textbf{s}}}}$ in $(1-u_1^{-1})(1-u_2^{-1})P^L_\emptyset$ then,
\[\widehat{HFL}(L,\textbf{s}) \cong \mathbb{F}^{|a_{\mbox{\scriptsize{\textbf{s}}}}|}_{s_1+s_2+\frac{\sigma -1}{2}}.\]
and therefore
\begin{equation}\label{eq:a4}
\widehat{HFL}(L,(1,n-1)) \cong \mathbb{F}_{(2n-2)}.
\end{equation}
Combining this with equation (\ref{eq:a3}) gives a contradiction, since $n$ is an integer.
\end{proof}

\begin{prop}
Suppose $L = b(q'k+1,k) = D(1,-1,\ldots,1,-1,w)$ is an $L$-space link, then $q' = 1$.
\end{prop}
\begin{proof}
We follow the same proof as the previous proposition.
First note that, in this case lk$(L_1,L_2) = -l_n = -w-n+1$; and so 
\[P^L_{L_1}(u_2) = (u_2)^{\frac{-w-n+1}{2}} \sum_{i=0}^\infty (u_2)^{-i} \mbox{ and } P^L_{L_2}(u_1) = (u_1)^{\frac{-w-n+1}{2}} \sum_{i=0}^\infty (u_1)^{-i}.\]
We can compute
\[\Delta_n(u_1,u_2) = \sum_{\substack{0\leq i \leq w-1\\0\leq j\leq n-1}} u_1^{i+j}u_2^{i-j+n-1} - \sum_{\substack{1 \leq i \leq w-1\\0 \leq j \leq n-2}}u_1^{i+j}u_2^{i-j+n-2},\]
which gives
\[P^L_\emptyset = -\sum_{\substack{0\leq i \leq w-1\\0\leq j\leq n-1}} u_1^{i+j+\frac{-w-n+3}{2}}u_2^{i-j+\frac{-w+n+1}{2}} + \sum_{\substack{1 \leq i \leq w-1\\0 \leq j \leq n-2}}u_1^{i+j+\frac{-w-n+3}{2}}u_2^{i-j+\frac{-w+n-1}{2}},\]
where the signs are fixed by corollary \ref{cor:alex2}. Using this, we compute $\mathfrak{T}(L)$ inside the region bounded by $s_1+s_2 \geq w-2$,$s_1\geq \frac{w-n+1}{2}$ and $s_2 \geq \frac{w-n+1}{2}$ and it is shown in Figure \ref{fig:case2}.

\begin{figure}[t]
    \centering
		\includegraphics{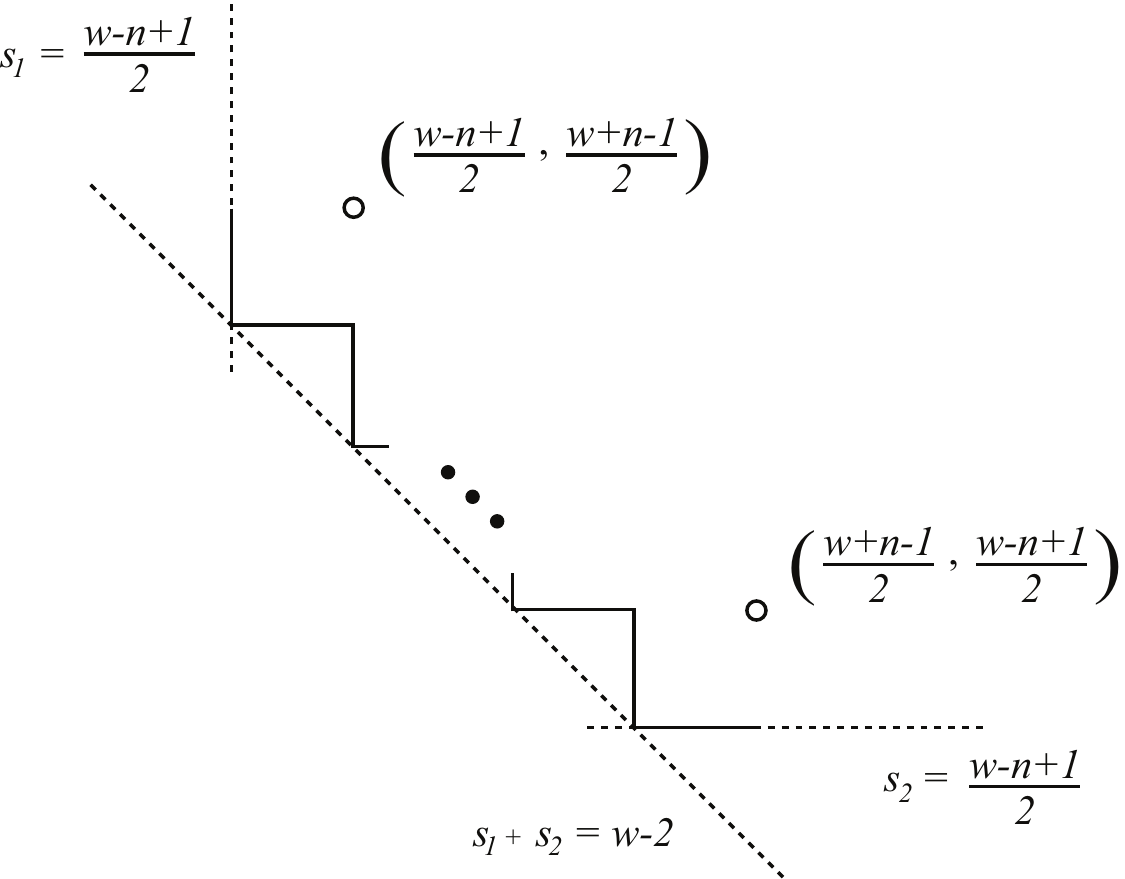}
    \caption{Part of $\mathfrak{T}(L)$ for $b(q'k+1,k)$, assuming it is an $L$-space link.}
		\label{fig:case2}
\end{figure}

\begin{equation}\label{eq:b1}
\widehat{HFL}\left(L,\left(\frac{w-n+1}{2},\frac{w+n-1}{2}\right)\right) \cong \mathbb{F}_{(1)}
\end{equation}
We can do this computation again using the fact that $L$ is alternating and to get
\begin{equation}\label{eq:b2}
\widehat{HFL}\left(L,\left(\frac{w-n+1}{2},\frac{w+n-1}{2}\right)\right) \cong \mathbb{F}_{(2w-1)}.
\end{equation}
combining equation \ref{eq:b1} and equation \ref{eq:b2} then gives $w = 1$ which implies $q' = 1$ as desired
\end{proof}
\begin{prop}
If $L = b(q'k+1,-k) = D(-1,1,\ldots,-1,1,-w)$ is an $L$-space link, then $k = 1$.
\end{prop}
\begin{proof}
Here lk$(L_1,L_2) = -l_n = w+n-1$, and so 
\[P_{L_1}^L(u_2) = (u_2)^{\frac{w+n-1}{2}}\sum_{i=0}^\infty (u_2)^{-i} \mbox{ and } P_{L_2}^L(u_1) = (u_1)^{\frac{w+n-1}{2}}\sum_{i=0}^\infty (u_1)^{-i}\]
and
\[P^{L}_\emptyset = -\sum_{\substack{1-w \leq i\leq -1\\0\leq j\leq n-2}}u_1^{i+j+\frac{w-n+3}{2}}u_2^{i-j+\frac{w+n-1}{2}} + \sum_{\substack{-w \leq i\leq -1\\0\leq j\leq n-1}}u_1^{i+j+\frac{w-n+3}{2}}u_2^{i-j+\frac{w+n+1}{2}}\]
where we have fixed signs for $P^{L}_\emptyset$, as in the previous two propositions using corollary \ref{cor:alex2}. Note that both edges going to $\left(\frac{w+n-1}{2},\frac{w+n-1}{2}\right)$ must be labeled with $1$ because they are determined by $P^L_{L_i}$ since $m(L) = \left(\frac{w+n-1}{2},\frac{w+n-1}{2}\right)$. Also notice that when $n>1$, the point $\left(\frac{w+n-1}{2},\frac{w+n-1}{2}\right)$ is outside of the Newton polytope for $P^L_\emptyset$. Thus both edges from $\left(\frac{w+n-3}{2},\frac{w+n-3}{2}\right)$ are also labeled with $1$. So we get
\[\widehat{HFL}\left(L,\left(\frac{w+n-1}{2},\frac{w+n-1}{2}\right)\right) \cong \mathbb{F}_{(0)} \oplus \mathbb{F}_{(-1)},\]
which is a contradiction because for an alternating link $L$ we know $\widehat{HFL}(L,\mbox{\textbf{s}})$ is only supported in one degree. Thus, we must have $n = 1$, which forces $k = 1$ as well.
\end{proof}
\begin{proof}[proof of Theorem \ref{thm:bridge}]
Combining the previous three propositions (also Lemma \ref{lem:polyres}) shows that, if $b(\alpha,\beta)$ is an $L$-space link, then it is either $b(qk-1,-k)$ for $q$ and $k$ odd positive integers, or of the form $b(k+1,k)$ where $k$ is odd. Note that reversing the orientation of one of the components of $b(k+1,k)$ gives $b(k+1,-1)$, which proves the Theorem.
\end{proof}
\nocite{*}
\bibliographystyle{plain}
\bibliography{NDarxiv}
\end{document}